\documentclass{amsart}
\usepackage{hyperref}
\usepackage[utf8]{inputenc}
\usepackage{graphicx}
\usepackage{color}
\usepackage{amsmath}% amstex ?
\usepackage{amsthm}
\usepackage{amssymb,bbm}%
\usepackage[numbers, square]{natbib}
\usepackage{mathrsfs}

\newcommand{\bN}{\mathbb{N}}

\newcommand{\bC}{\mathbb{C}}

\newcommand{\bR}{\mathbb{R}}

\newcommand{\bD}{\mathbb{D}}

\newcommand{\cD}{\mathcal{D}}

\newcommand{\Herm}{\mathrm{He}}

\DeclareMathOperator*{\argmax}{arg\,max}
\DeclareMathOperator*{\dist}{dist}

\newcommand{\E}{\mathbb E}

\newcommand{\R}{\mathbb{R}}
\newcommand{\N}{\mathbb{N}}

\newcommand{\C}{\mathbb{C}}
\newcommand{\Z}{\mathbb{Z}}

\renewcommand{\P}{\mathbb{P}}

\renewcommand{\Re}{\operatorname{Re}}
\renewcommand{\Im}{\operatorname{Im}}

\newcommand{\Var}{\mathop{\mathrm{Var}}\nolimits}

\newcommand{\eps}{\varepsilon}

\newcommand{\toas}{\overset{a.s.}{\underset{n\to\infty}\longrightarrow}}

\newcommand{\ton}{\overset{}{\underset{n\to\infty}\longrightarrow}}

\newcommand{\ind}{\mathbbm{1}}

\newcommand{\dd}{{\rm d}}
\newcommand{\eee}{{\rm e}}

\theoremstyle{plain}
\newtheorem{theorem}{Theorem}[section]
\newtheorem{lemma}[theorem]{Lemma}
\newtheorem{corollary}[theorem]{Corollary}
\newtheorem{proposition}[theorem]{Proposition}

\theoremstyle{definition}

\newtheorem{example}[theorem]{Example}

\theoremstyle{remark}
\newtheorem{remark}[theorem]{Remark}

\begin{document}

\author{Rudolf Gr\"ubel}
\address{Rudolf Gr\"ubel, Institut f\"ur Mathematische Stochastik,
Leibniz Universit\"at Hannover,
Welfengarten 1,
30167 Hannover, Germany
}
\email{rgrubel@stochastik.uni-hannover.de}

\author{Zakhar Kabluchko}
\address{Zakhar Kabluchko, Institut f\"ur Mathematische Statistik,
Universit\"at M\"unster,
Orl\'eans--Ring 10,
48149 M\"unster, Germany}
\email{zakhar.kabluchko@uni-muenster.de}

\title[Edgeworth expansions for profiles of lattice BRW's]{Edgeworth expansions for profiles of lattice branching random walks}
\keywords{Branching random walk, Edgeworth expansion, central limit theorem, profile, Biggins martingale, random analytic function, mod-$\varphi$-convergence, height, mode}

\subjclass[2010]{Primary, 60G50; secondary, 60F05, 60J80, 60F10, 60F15}
\thanks{}
\begin{abstract}
Consider a branching random walk on $\Z$ in discrete time. Denote by $L_n(k)$ the number of particles at site $k\in\Z$ at time $n\in\N_0$.
By the \textit{profile} of the branching random walk (at time $n$) we mean the function
$k\mapsto L_n(k)$. We establish the following asymptotic expansion of $L_n(k)$, as $n\to\infty$:
$$
\eee^{-\varphi(0)n} L_n(k) = \frac{\eee^{-\frac 12 x_n^2(k)}}{\sqrt {2\pi \varphi''(0) n}}
                              \sum_{j=0}^r \frac{F_j(x_n(k))}{n^{j/2}} + o\left(n^{-\frac{r+1}{2}}\right) \text{ a.s.},
$$
where $r\in\N_0$ is arbitrary,  $\varphi(\beta)=\log \sum_{k\in\Z} \eee^{\beta k}  \E L_1(k)$ is the cumulant generating function of the intensity of the branching random walk and
$$
    x_n(k) = \frac{k-\varphi'(0) n}{\sqrt{\varphi''(0)n}}. % \quad \mu=\varphi'(0), \quad \sigma^2=\varphi''(0).
$$
The expansion is valid uniformly in $k\in\Z$ with probability $1$ and the $F_j$'s are polynomials whose random
coefficients can be expressed through the derivatives of $\varphi$ and the derivatives of the limit of the Biggins
martingale at $0$. Using exponential tilting, we
also establish
more general expansions covering the whole range of the branching random walk except its extreme values. As an
application of this expansion for $r=0,1,2$ we recover in a unified way a number of known results and establish
several
new limit theorems.  In particular, we study the a.s.\ behavior of the individual occupation numbers $L_n(k_n)$, where $k_n\in\Z$ depends on $n$ in some regular way. We also prove a.s.\ limit theorems for the mode $\argmax_{k\in\Z} L_n(k)$ and the height $\max_{k\in\Z} L_n(k)$ of the profile. The asymptotic behavior of these quantities depends on whether the drift parameter $\varphi'(0)$ is integer, non-integer rational, or irrational. Applications of our results to profiles of random trees including binary search trees and random recursive trees will be given in a separate paper.

\end{abstract}

\maketitle

\section{Introduction}

\subsection{Statement of the problem}\label{subsec:statement_problem}
The \textit{branching random walk} (BRW) is a model that combines random spatial motion of particles with branching; see~\cite{biggins_branching_out} for a historical overview. In this paper, we restrict our attention to the BRW on the \textit{integer lattice $\Z$}.  The model is defined as follows.
%The state of the branching random walk at time $n$ is described by a point process $\pi_n$ which records the positions of particles at time $n$.
At time $0$ consider a single ancestor particle located at $0$. At any time $n\in\N_0$ every particle is replaced (independently of all other particles and of the past of the process) by a random finite cluster of descendant particles whose displacements w.r.t.\ the original particle are distributed according to some fixed point process $\zeta$ on $\Z$. The number of particles in $\zeta$ may be random. The positions of the particles in $\zeta$ need not be independent random variables (nor need they be independent of the number of particles).

Denote by $N_n$ the number of particles at time $n\in \N_0$. Then $\{N_n\colon n\in\N_0\}$ is a Galton--Watson branching process. We will always assume that this process is supercritical (meaning that $m:=\E N_1>1$) and with probability $1$ never dies out (meaning that $N_1\neq 0$ a.s.) The latter assumption could be removed, but then all results hold conditionally on non-extinction. Denote the positions of the particles in the BRW at time $n$ by
$$
z_{1,n}\leq \ldots \leq z_{N_n,n}.
$$
Our main object of interest is the \textit{occupation number} $L_n(k)$ defined as the number of particles located at site $k\in\Z$ at time $n\in\N_0$:
\begin{equation}\label{eq:L_T_k_def_original}
L_n(k) =  \#\{1\leq j\leq N_n \colon z_{j,n} = k\}.
\end{equation}
The random function $L_n:\Z\to \N_0$ will be referred to as the \textit{profile} of the branching random walk at time $n$. The \textit{intensity} of the branching random walk at time $n$ is the measure $\vartheta_n$ on $\Z$ defined as the expectation of the profile:
\begin{equation}\label{eq:vartheta}
\vartheta_n(\{k\}) = \E L_n(k), \quad k\in\Z.
\end{equation}

We will study the asymptotic shape of the profile $L_n$ as $n\to\infty$. More concretely, we will obtain  an asymptotic expansion of $L_n$ in powers of $n^{-1/2}$ which is similar to the classical \textit{Chebyshev--Edgeworth--Cram\'er expansion} for sums of independent identically distributed random variables. The latter will be recalled in Section~\ref{subsec:edgeworth_iid}.

%The convolution product of two finite measures $\alpha$ and $\beta$ on $\bZ$ assigns the value $\sum_{k\in\bZ}\alpha(\{k\})\beta(\{n-k\})$ to the set $\{n\}$.
It follows from the definition of the branching random walk that the intensity measure $\vartheta_n$ is the $n$-fold convolution of $\vartheta_1$ (so that, in particular, $\vartheta_n(\Z)=m^n$).
Throughout the history of the BRW this fact has been used to relate the asymptotic properties of the random measures represented by $L_n$, as $n\to\infty$, to the analogous classical properties of convolutions of probability measures. In connection with the central limit theorem, this leads to the Harris conjecture~\cite[Chapter III, \S 16]{harris_book} which was proved (in various forms and for various models) in~\cite{stam,joffe_moncayo,asmussen_kaplan1,asmussen_kaplan2,biggins_CLT,uchiyama,biggins_uniform,yoshida}.  The work of~\citet{biggins_chernov,biggins_growth_rates,biggins_uniform} refers to this analogy in connection with large deviation principles and local limit theorems.
%$\vartheta_n$ shares many classical properties known for sums of i.i.d.\ random variables  In a series of works, Biggins~\cite{biggins_chernov,biggins_mart_conv,biggins_growth_rates,biggins_CLT,biggins_uniform} developed a philosophy according to which these properties of the intensity measure $\vartheta_n$ are in some or other form inherited by the profile $L_n$.

The present work adds one more example to this list by proving an asymptotic expansion of the profile. Since the asymptotic expansion contains many other classical limit theorems as consequences, we will be able to recover many of the above mentioned results in a unified way (though under sub-optimal moment conditions).
Our asymptotic expansion of the profile $L_n$ will be stated in Theorem~\ref{theo:asympt_expansion_BRW}. For comparison, an asymptotic expansion of the intensity measure $\vartheta_n$ will be given in Proposition~\ref{prop:expansion_expected}. It turns out that the two expansions do not coincide: while the expansion of the intensity $\vartheta_n$ is deterministic, the expansion of the profile $L_n$ contains random terms which can be expressed through the derivatives of a remarkable random analytic function given as the limit of the Biggins martingale; see below.

Our motivation for investigating such asymptotic expansions was to develop a method to study binary search trees, random recursive trees, and some similar types of random trees that appear in the analysis of algorithms. An important characteristic of these trees is their profile, that is the random function $k\mapsto L_n(k)$ counting the number of nodes at a given level $k$, for a tree with $n$ nodes. The probabilistic properties of these profiles have been much studied; see, e.g.,\ \cite{drmota_book,drmota_hwang,drmota_janson_neininger,chauvin_etal}. Since random trees can be embedded into continuous-time branching random walks, see~\citet{chauvin_etal} and~\citet{biggins_grey}, it is possible to translate our asymptotic expansion into the setting of random trees.  In Sections~\ref{subsec:xia_chen} and~\ref{subsec:width_mode}, we will answer the BRW analogues of several open questions on the profiles of random trees. For example, we will prove limit theorems on the height and the mode of the branching random walk.
While we plan to give applications to random trees in a separate paper, we will point out connections to the existing literature on random trees here.

%Our plan is as follows.
%\vspace*{2mm}
%\noindent
%\textit{The paper is organized as follows.}
The paper is organized as follows.
In Section~\ref{subsec:edgeworth_iid} we recall the classical asymptotic expansion for sums of i.i.d.\ random variables. In Section~\ref{subsec:BRW_def_assumpt} we fix the notation and state our assumptions on the BRW. In Section~\ref{subsec:simulations} we briefly comment on our simulations.
%In Section~\ref{subsec:known_results} we state some known results on the BRW which will be recovered and strengthened using our asymptotic expansion.
In Section~\ref{sec:results} we state the asymptotic expansion and its numerous consequences. Proofs are given in Section~\ref{sec:proofs}.

\subsection{Asymptotic expansion for sums of i.i.d.\ random variables}\label{subsec:edgeworth_iid}
Let $Z_1,Z_2,\ldots$ be independent identically distributed (i.i.d.)\ random variables with $\E Z_1=\mu$ and $\Var Z_1 = \sigma^2 > 0$. Define the sequence of their partial sums:
$$
S_n=Z_1+\ldots+Z_n.
$$
%In this paper, we will be mostly considering the case of integer-valued $Z_1$.
We assume that $Z_1$ takes only integer values and that the lattice span of the distribution of $Z_1$ is $1$, that is there is no pair $h\in \{2,3,\ldots\}$, $a\in\Z$ such that all possible values of $Z_1$ are contained in the arithmetic progression $a+ h\Z$.
%$$
%\gcd \{k-l\colon k,l\in I\} = 1, \text{ where } I=\{k\in\Z\colon \P[Z_1=k]>0\}.
%$$

Under these assumptions, the \textit{local limit theorem} (see, e.g.\ Theorem~1 in~\cite[Ch.~VII, p.~187]{petrov_book}) states that
\begin{equation}
\lim_{n\to\infty} \sqrt n \sup_{k\in\Z} \left|\P[S_n=k] - \frac{1}{\sqrt {2\pi n}\, \sigma} \eee^{-\frac 12 x_n^2(k)} \right| = 0,
\end{equation}
where
\begin{equation}
x_n(k)=\frac{k- n\mu}{\sigma \sqrt n},\quad k\in\Z.
\end{equation}

Under additional moment conditions, there is a complete asymptotic expansion of $\P[S_n=k]$ in powers of $n^{-1/2}$, the classical Chebyshev--Edgeworth--Cram\'er expansion. Various versions of this expansion have been much studied; see the monographs by~\citet{petrov_book}, \citet{bhattacharya_ranga_rao_book}, \citet{hall_book}. To state the expansion relevant to us, suppose additionally that $\E |Z_1|^{r+2}<\infty$ for some $r\in\N_0$. The logarithm of the characteristic function of $Z_1$ can then be written in the form
$$
\log \E [\eee^{is Z_1}] = \sum_{j=1}^{r+2} \kappa_j \frac{(is)^j}{j!} + o(|s|^{r+2}), \text{ as } s\to 0,
$$
where the numbers $\kappa_j$ are the \textit{cumulants} of $Z_1$. Note that $\kappa_1=\mu$ and $\kappa_2=\sigma^2$. The Chebyshev--Edgeworth--Cram\'er asymptotic expansion reads as follows, see Theorem 13 in~\citet[Ch.~VII, p.~205]{petrov_book}:
%\begin{theorem}
%Under the above assumptions,
\begin{equation}\label{eq:asympt_expansion_random_walk}
\lim_{n\to\infty} n^{\frac {r+1}2} \sup_{k\in\N} \left|\P[S_n=k] - \frac 1 {\sqrt {2\pi n}\,\sigma} \eee^{-\frac 12 x^2_n(k)}
\sum_{j=0}^{r}
\frac 1 {n^{j/2}}q_j(x_n(k))\right|=0,
\end{equation}
where $q_j(x)$ is a degree $3j$ polynomial in $x=x_n(k)$ whose coefficients can be expressed through the cumulants $\kappa_2,\ldots,\kappa_{j+2}$.
%\end{theorem}
The first three terms in the expansion are given by
\begin{align}
q_0(x) = 1,
\quad
q_1(x) = \frac{\kappa_3}{6\sigma^3} \Herm_3(x),
\quad
q_2(x) = \frac{\kappa_4}{24 \sigma^4} \Herm_4(x) + \frac{\kappa_3^2}{72 \sigma^6} \Herm_6(x),
\end{align}
where  $\Herm_n(x)$ denotes the $n$-th ``probabilist'' \textit{Hermite polynomial}:
$$
\Herm_n(x)= \eee^{\frac 12 x^2} \left(-\frac{\dd}{\dd x}\right)^n \eee^{-\frac 12 x^2}.
$$
%where $D=$ is the differentiation operator.
The first few Hermite polynomials relevant to us are
\begin{align}
&\Herm_1(x)= x,\quad
\Herm_2(x)= x^2-1, \quad
\Herm_3(x) = x^3-3x,\label{eq:Herm1}\\
&\Herm_4(x)= x^4-6x^2+3, \quad
\Herm_6(x)= x^6 - 15 x^4 + 45 x^2 - 15. \label{eq:Herm2}
\end{align}
%Similar asymptotic expansions exist for the density of $S_n$ if one assumes that $Z_1$ is absolutely continuous, and for the distribution function of $S_n$ if one imposes a strong non-lattice assumption on $Z_1$.
The aim of the present work is to obtain asymptotic expansions of this type for the profiles of branching random walks.  %Let us define the model we are interested in.

\subsection{Assumptions on the branching random walk}\label{subsec:BRW_def_assumpt}
%Let
%$$
%\pi_n=\sum_{j=1}^{N_n} \delta_{z_{j,n}}
%$$
%be the point process recording their positions.
%The only parameter needed to identify the law of the discrete-time BRW is the law of the point process $\zeta$ encoding the shifts of the offspring particles w.r.t.\ their parent.
Consider a branching random walk on the integer lattice $\Z$, as defined in Section~\ref{subsec:statement_problem}.
We will use the following standing assumptions.
Recall that $\vartheta_1$ denotes the intensity measure of the BRW at time $n=1$.

\vspace*{2mm}
\noindent
\textsc{Assumption A:} The branching random walk is non-degenerate, that is the set $\{k\in\Z\colon \vartheta_1(\{k\})>0\}$ contains at least two elements.
\vspace*{2mm}

Recall that $N_n$ denotes the number of particles in the BRW at time $n$.

\vspace*{2mm}
\noindent
\textsc{Assumption B:} $N_1\neq 0$ a.s.\ and $m := \E N_1 \in (1,\infty)$.
\vspace*{2mm}

Assumption~A simply excludes a trivial case. As we already mentioned in Section~\ref{subsec:statement_problem},
the first part of Assumption~B could be removed, but then all results hold conditionally on non-extinction.

An important role will be played by the \textit{cumulant generating function} $\varphi$ of the intensity $\vartheta_1$:
\begin{equation}\label{eq:def_varphi}
\varphi(\beta) = \log \E \left[\sum_{k\in\Z} \eee^{\beta k}  L_1(k)\right] \in (-\infty,+\infty],
\quad \beta\in\R.
\end{equation}
Clearly, $\varphi(0)=\log m>0$.

\vspace*{2mm}
\noindent
\textsc{Assumption C:} The function $\varphi$ is finite in some open interval containing $0$.

\vspace*{2mm}
Let $\cD_\varphi$  be the maximal open interval on which $\varphi$ is finite. It follows that $\varphi$ is strictly convex and infinitely differentiable on $\cD_\varphi$.
A crucial role in the study of the branching random walk is played by the \textit{Biggins martingale}:
\begin{equation}\label{eq:biggins_martingale_def}
W_{n}(\beta)
%= \eee^{-\varphi(\beta)n} \int \eee^{\beta z} \pi_1(\dd z)
:=
\eee^{-\varphi(\beta)n} \sum_{i=1}^{N_n} \eee^{\beta z_{i,n}}
=
\sum_{k\in\Z}  \eee^{\beta k -\varphi(\beta)n} L_n(k),
\;\;\; \beta\in\cD_\varphi.
\end{equation}
By the martingale convergence theorem for non-negative martingales, we have for all $\beta\in \cD_\varphi$,
\begin{equation}\label{eq:W_n_W_beta}
W_n(\beta) \toas W_\infty(\beta).
\end{equation}
It is, however, possible that $W_\infty(\beta)=0$ a.s. The range of $\beta$ where this does not happen was found by Biggins~\cite{biggins_mart_conv}. Denote by $(\beta_-,\beta_+)\subset \cD_\varphi$ the open interval on which $\varphi'(\beta)\beta <\varphi(\beta)$:
\begin{align}
\beta_-&=\inf\{\beta\in\cD_\varphi\colon \varphi'(\beta) \beta < \varphi(\beta)\}, \label{eq:beta_-}\\
\beta_+&=\sup\{\beta\in\cD_\varphi\colon \varphi'(\beta) \beta < \varphi(\beta)\}. \label{eq:beta_+}
\end{align}
Clearly, $0\in (\beta_-,\beta_+)$, so that this interval is non-empty. The endpoints of the intervals $\cD_\varphi$ and $(\beta_-,\beta_+)$ are allowed to be infinite.  We also need the following moment condition which supplements Assumption~C.

\vspace*{2mm}
\noindent
\textsc{Assumption D:} There is a $p > 1$  such that for every compact set $K\subset (\beta_-,\beta_+)$,
\begin{equation}\label{eq:standing_assumption}
%\sup_{\beta\in K}\E \left[\left(\sum_{k\in\Z} \eee^{\beta k} L_1(k)  \right)^{p}\right]<\infty.
\sup_{\beta\in K}\E \left[W_1^{p}(\beta)\right]<\infty.
\end{equation}
Under the above assumptions, it follows from \citet[Theorem~A]{biggins_mart_conv} that for all $\beta\in (\beta_-,\beta_+)$,
$$
W_\infty(\beta)>0 \text{ a.s.}, \;\;\; \E W_\infty(\beta)=1.
$$

It is a crucial observation due to~\citet{biggins_uniform_IMS,biggins_uniform}
and~\citet{uchiyama} that the martingale convergence~\eqref{eq:W_n_W_beta} can be extended to a \textit{complex} neighborhood of the interval $(\beta_-,\beta_+)$. First of all, it is clear that the function $\varphi$ is defined as an analytic function of $\beta$ in a sufficiently small open set $U\subset \C$ containing the interval $\cD_\varphi$. It follows that for every $n\in\N$, the function $W_n(\beta)$ is well-defined as an analytic function on $U$. It has been shown in~\cite{biggins_uniform} and~\cite{uchiyama} that under Assumption~D there exist an open set $V\subset \C$ containing the interval $(\beta_-,\beta_+)$ and a random analytic function $W_\infty(\beta)$ on $V$ such that
\begin{equation}
\sup_{\beta \in K} |W_\infty(\beta) - W_n(\beta)| \toas 0
\end{equation}
for every compact set $K\subset V$.

Assumption~D is essential for our method of proof; see also Section~\ref{sec:proof_exp_char}
and the discussion in~\cite{biggins_uniform_IMS}.

Finally, as in the classical Chebyshev--Edgeworth--Cram\'er expansion, we need to assume that the lattice width
associated with the support of the intensity measure $\vartheta_1$ equals $1$.

\vspace*{2mm}
\noindent
\textsc{Assumption E:} There is no pair $h\in \{2,3,\ldots\}$, $a\in\Z$ such that $\vartheta_1$ is concentrated on the arithmetic progression $a+h\Z$.

\vspace*{2mm}
Assumption~E is not a restriction of generality because it can always be achieved by a suitable affine transformation; see Example~\ref{ex:xia_chen} below.

\subsection{Simulations}\label{subsec:simulations}
At various places we illustrate our results by simulations. For Figures~\ref{bild:BRW_CLT}, \ref{bild:BRW_LDP} and Figure~\ref{bild:BRW_L_n_k_n} (top) these are based on the point process $\zeta$ given by
$$
\P[\zeta = \delta_0] = \P[\zeta = \delta_{-1}+\delta_{+1}] = \frac 12,
$$
where $\delta_z$ is the Dirac delta-measure at $z\in\Z$.
In our simulations of the BRW we do not keep track of the individual locations as the number of particles grows
exponentially. Instead we make use of the following rules to obtain $L_n$:  $L_0(k)= \ind_{\{k=0\}}$ and
\begin{align*}
L_{n+1}(k) = (L_n(k) - T_n(k)) + T_n(k-1) + T_n(k+1), \;\;\; n\in\N_0, \; k\in \Z,
\end{align*}
where $T_n(k)$ is the number of particles at time $n$ and site $k$ that have $2$ descendants at time $n+1$ (located at sites $k+1$ and $k-1$). Given $L_n(k)$, $k\in\Z$, the random variables $T_n(k)$, $k\in\Z$, are conditionally independent and $T_n(k) \sim \text{Bin} (L_{n}(k), 1/2)$.

For Figure~\ref{bild:BRW_L_n_k_n} (middle, bottom) and Figures~\ref{bild:BRW_Argmax}, \ref{bild:BRW_Max} the simulations are based on the point process $\zeta$ with
$$
\P[\zeta = \delta_0] = p, \;\;\; \P[\zeta = 2\delta_{+1} + \delta_{-1}] = 1 - p
$$
with $p = 1/2$ (rational case) and $p =3(\pi-3)$ (irrational case).

\section{Results}\label{sec:results}

\subsection{Asymptotic expansion of the profile}
Consider a branching random walk on $\Z$ which satisfies Assumptions A--E. Recall that $L_n(k)$ denotes the number of particles of the branching random walk that are located at site $k\in\Z$ at time $n\in\N_0$.
%\begin{equation}\label{eq:L_T_k_def}
%L_n(k) = \pi_n(\{k\}).
%\end{equation}
Take some $\beta\in(\beta_-,\beta_+)$. Our limit theorems will be stated in terms of the ``tilted'' profile
$$
k\mapsto \eee^{\beta k -\varphi(\beta)n}L_n(k), \quad k\in \Z.
$$
As known from large deviations theory, the tilting operation allows to better access the properties of the profile for $k\approx \varphi'(\beta)n$.  Define the corresponding tilted cumulant
\begin{equation}
\kappa_j(\beta)=\varphi^{(j)}(\beta)
\end{equation}
as the $j$-th derivative of $\varphi$ at $\beta$, $j\in\N$.
%so that for sufficiently small $|u|$,
%\begin{equation}
%\varphi(\beta+u)-\varphi(\beta)  = \sum_{j=1}^{\infty} \kappa_j(\beta) \frac{u^j}{j!}.
%\end{equation}
In particular, we need the notation
\begin{equation}
\mu(\beta)  := \varphi'(\beta) = \kappa_1(\beta), \quad \sigma ^2(\beta) := \varphi''(\beta) = \kappa_2(\beta) > 0.
\end{equation}
Introduce the ``standardized coordinate''
\begin{equation}\label{eq:x_n_k_def}
x = x_n(k) := \frac{k-\mu(\beta) n}{\sigma(\beta) \sqrt n}, \quad k\in\Z.
\end{equation}
Note that for ease of reading we omit the dependence on $\beta$ in $x_n(k)$.
%Recall that for $k\in\Z$ we defined
%\begin{equation}
%x = x_n(k) = \frac{k-\mu n}{\sigma \sqrt n}.
%\end{equation}

Our first theorem gives an asymptotic expansion of the tilted occupation number
$\eee^{\beta k - \varphi(\beta )n} L_n(k)$ in powers of $n^{-1/2}$.

\begin{theorem}\label{theo:asympt_expansion_BRW}
Consider a branching random walk satisfying Assumptions A--E. Fix $r\in\N_0$ and a compact interval $K\subset (\beta_-,\beta_+)$. Then we have
\begin{equation}\label{eq:asympt_expansion_BRW}
n^{\frac {r+1}2}
\sup_{\beta\in K}\sup_{k\in\Z}\left|\eee^{\beta k - \varphi(\beta )n} L_n(k)-
\frac {\eee^{-\frac 12 x^2_n(k)}} {\sigma(\beta) \sqrt {2\pi n}}
\sum_{j=0}^{r}
\frac 1 {n^{j/2}}F_j(x_n(k);\beta )\right| \toas 0
,
\end{equation}
where $F_{j}(x;\beta)$ is a degree $3j$ polynomial in $x$.
The coefficients of $F_{j}(x;\beta )$ are random and can be expressed through
%the cumulants
$$
\kappa_2(\beta ), \ldots, \kappa_{j+2}(\beta ) \text{ and } W_\infty(\beta ), W_\infty'(\beta ), \ldots, W_{\infty}^{(j)}(\beta ).
$$
%and the derivatives of $W_\infty$ at $\beta $ of orders $0,\ldots,j$.
\end{theorem}
Stating the complete  formula for $F_j(x;\beta )$ requires introducing complicated notation and is therefore postponed
to the proof in Section~\ref{sec:prooftheo2.1}. Here, we provide only the first three terms of the expansion:
\begin{align}
F_0(x;\beta ) &= W_\infty(\beta ),\label{eq:F_0_BRW}\\
F_1(x;\beta ) &= W_\infty(\beta ) \frac{\kappa_3(\beta )}{6\sigma^3(\beta)} \Herm_3(x) + W_\infty' (\beta ) \frac{x}{\sigma(\beta)},\label{eq:F_1_BRW}\\
F_2(x;\beta ) &= W_\infty(\beta )\left(\frac{\kappa_4(\beta )}{24 \sigma^4(\beta)} \Herm_4(x) + \frac{\kappa_3^2(\beta )}{72 \sigma^6(\beta)} \Herm_6(x)\right)\label{eq:F_2_BRW}\\
&+ W_\infty'(\beta ) \frac{\kappa_3(\beta )}{6\sigma^4(\beta)}\Herm_4(x)
+ W_\infty''(\beta ) \frac{1}{2\sigma^2(\beta)}\Herm_2(x).\notag
\end{align}
\begin{remark}
We will see in the proof of Theorem~\ref{theo:asympt_expansion_BRW}, Equation~\eqref{eq:U_r_n_def}, that
$$
F_j(x;\beta) = \sum_{m=0}^j \frac{W_{\infty}^{(m)}(\beta)}{m!} Q_{m,j}(x;\beta),
$$
where $Q_{0,j}(x;\beta), \ldots, Q_{j,j}(x;\beta)$ are deterministic polynomials in $x$ with coefficients depending on $\beta$. We will also see that $Q_{0,j}(x;\beta)$ (the coefficient of $W_\infty(\beta)$) is the same as  $q_j(x)$ in~\eqref{eq:asympt_expansion_random_walk} but with $\kappa_j$ replaced by $\kappa_{j}(\beta)$.
\end{remark}

\begin{remark}\label{rem:polynomial_order}
Theorem~\ref{theo:asympt_expansion_BRW} remains valid if in the formula for $F_j$ we replace the derivatives of
$W_\infty$ by the corresponding derivatives of $W_n$. The reason is that w.p.\ $1$, $W_n$ converges to $W_\infty$
exponentially fast together with all its derivatives
 (see Lemma~\ref{lem:W_infty_W_T}), while the error term in Theorem~\ref{theo:asympt_expansion_BRW}
is of polynomial order only. This fact is used in our simulations where we replace $W_\infty$ by $W_n$. It was shown in~\cite{roesler_topchii_vatutin1}, see also~\cite{neininger}, that in a suitable range of $\beta$ the asymptotic distribution of the appropriately normalized difference $W_n(\beta) - W_{\infty}(\beta)$ is a mixture of centered normals. A functional limit theorem for this difference was obtained in~\cite{gruebel_kabluchko}.
%We do not use this refined asymptotics  in our proofs because the difference $W_\infty(\beta)-W_n(\beta)$ is exponentially small while our expansions are in powers of $n^{-1/2}$; see also Remark~\ref{rem:polynomial_order}.
\end{remark}

\begin{remark}
For the branching Brownian motion, an expansion similar to~\eqref{eq:asympt_expansion_BRW} (with $\beta=0$) was obtained by~\citet{revesz_etal}. In a general BRW, the transition mechanism is not Gaussian, so that methods specific to the Gaussian setting cannot be used.
\end{remark}

In the rest of Section~\ref{sec:results} we state numerous consequences of Theorem~\ref{theo:asympt_expansion_BRW}. It
allows us to recover many known results in a unified way and to answer a number of open questions on the individual occupation numbers $L_n(k)$, as well as the height and the mode of the BRW profile.

\subsection{Local and global central limit theorems}
Taking only the first term in the expansion given in Theorem~\ref{theo:asympt_expansion_BRW} (meaning that $r=0$) we obtain
\begin{equation}
\sqrt n \sup_{k\in\Z} \left|\eee^{\beta k - \varphi(\beta) n} L_n(k) -
\frac {W_\infty(\beta)} {\sqrt {2\pi \varphi''(\beta) n}} \exp\left\{-\frac {(k-\varphi'(\beta) n)^2}{2\varphi''(\beta) n}\right\}
\right| \toas 0.
\end{equation}
%Roughly speaking, this result says that the tilted profile $k\mapsto \eee^{\beta k - \varphi(\beta ) n} L_n(k)$ has approximately Gaussian shape with mean $\mu(\beta) n$, standard deviation $\sigma(\beta)\sqrt n$, and the total mass of the tilted profile is approximately $W_\infty(\beta )$.
The most interesting case is $\beta = 0$ when there is no tilting and we obtain the following  local limit theorem for the BRW:
\begin{equation}\label{eq:local_limit_BRW}
\sqrt n \sup_{k\in\Z} \left| \frac{L_n(k)}{m^{n}} -
\frac {W_\infty(0)} {\sqrt {2\pi\varphi''(0) n}} \exp\left\{-\frac {(k-\varphi'(0)n)^2}{2\varphi''(0) n}\right\}
\right| \toas 0.
\end{equation}
%This result be seen as a local limit theorem for the branching random walk.
Roughly speaking, \eqref{eq:local_limit_BRW} says that the profile $k\mapsto L_n(k)$ has an approximately Gaussian shape
with mean $\varphi'(0) n$, standard deviation $\sqrt {\varphi''(0) n}$, and the total mass of the profile is $N_n \sim
W_\infty(0)m^n$; see the left part of Figure~\ref{bild:BRW_CLT}. In Figures~\ref{bild:BRW_CLT} and~\ref{bild:BRW_LDP}
we use vertical bars for the discrete profiles and continuous lines for the approximating functions.

\begin{figure}[!htbp]
\includegraphics[width=0.49\textwidth]{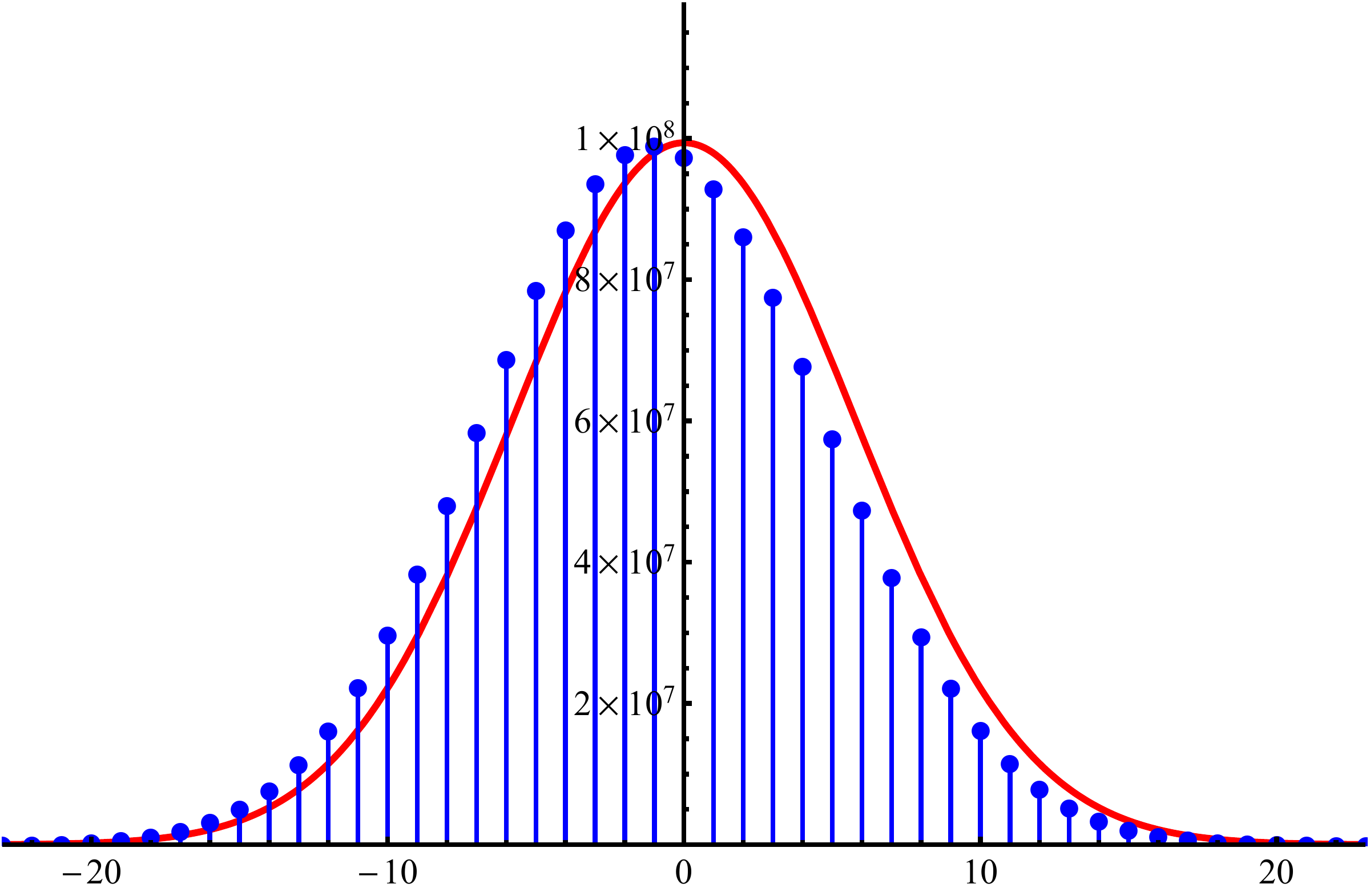}
\includegraphics[width=0.49\textwidth]{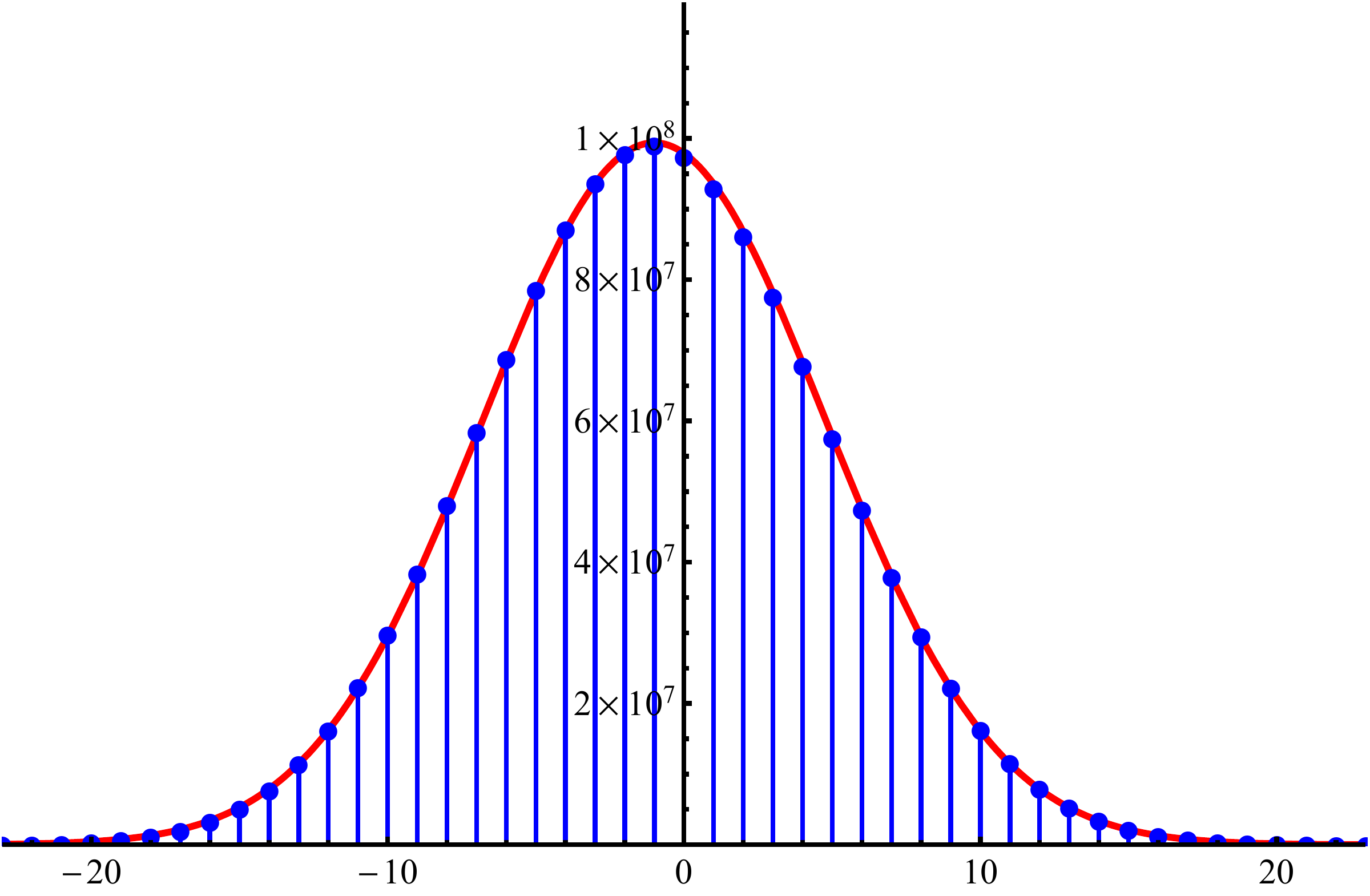}
\caption
{\small Left: A realization of the profile $k\mapsto L_n(k)$ of a branching random walk (blue) at time $n=50$ and the first, Gaussian term (see~\eqref{eq:local_limit_BRW}) in the asymptotic expansion (red). Right: The same profile (blue) and the Gaussian term shifted by $W'_n(0)/W_n(0)$ (red).}
\label{bild:BRW_CLT}
\end{figure}

Let us now look at the next term in the expansion. Theorem~\ref{theo:asympt_expansion_BRW} with $r=1$ and $\beta=0$ yields that
\begin{multline}\label{eq:local_shift_corr}
\frac{L_n(k)}{W_\infty(0) m^n}  - \frac {\eee^{-\frac 12 x_n^2(k)}} {\sigma(0) \sqrt {2\pi n}}
\\
=\frac {\eee^{-\frac 12 x_n^2(k)}} {\sigma(0) \sqrt {2\pi n}} \cdot \frac 1{\sqrt n} \left(\frac{\kappa_3(0)}{6\sigma^3(0)} \Herm_3(x_n(k)) + \frac{W_\infty'(0)}{W_\infty(0)} \frac {x_n(k)}{\sigma(0)}  \right) + o\left(\frac{1}{n}\right)
\quad \text{a.s.},
\end{multline}
where the $o$-term is uniform in $k\in\Z$.
There are two correction terms on the right-hand side of~\eqref{eq:local_shift_corr}. The term involving $\Herm_3$ is the ``shape correction'' to the Gaussian profile. The same term appears in the expansion of the expected profile $\E L_n(k)$; see Section~\ref{subsec:expansion_expected}. The term involving $W_\infty'(0)/W_\infty(0)$  can be thought of as a random ``location correction''. Indeed, this term says that in order to obtain a better approximation  to the BRW profile we have to take a Gaussian profile centered at $\varphi'(0)n + W_\infty'(0)/W_\infty(0)$ rather than at $\varphi'(0)n$; see the left and right parts of Figure~\ref{bild:BRW_CLT}.

%Introduce the ``shape corrected'' Gaussian density
%$$
%p_n(y;\mu,\tau^2) = \frac{1}{\sqrt{2\pi} \tau} \eee^{-\frac{(y-\mu)^2}{2\tau^2}} \left(1 + \frac 1{\sqrt n} %\frac{\kappa_3(0)}{6\sigma^3(0)}\Herm_3(y)\right).
%$$
%Then, Theorem~\ref{theo:asympt_expansion_BRW} with $r=1$ and $\beta=0$ yields that
%$$
%\sup_{k\in\Z} n \left|\frac{L_n(k)}{m^n} - W_\infty(0) p_n\left(k; \varphi'(0) n+ \frac{W'_\infty(0)}{W_\infty(0)}, \varphi''(0) n\right)\right| \toas %0.
%$$
%That is, the profile can be approximated by the Gaussian density with variance $\varphi''(0)n$ centered at $\varphi'(0)n + %\frac{W_\infty'(0)}{W_\infty(0)}$.

Note that the random variable $W_\infty'(0)$ appearing above is the a.s.\ limit of the martingale
$$
W_n'(0)  = \frac 1 {m^n} \sum_{i=1}^{N_n} (z_{i,n} - \varphi'(0)n).
$$
Since the total number of particles at time $n$ is $N_n = W_n(0) m^n$, we can view $W_n'(0)/W_n(0)$ as an estimate for the ``shift'' of the profile w.r.t.\ its ``expectation'' $\varphi'(0)n$. This explains the appearance of $W_\infty'(0)/W_\infty(0)$ as a ``location correction'' in a quite natural way.
Similarly, the variable $W_\infty''(0)$ which will appear frequently below is the limit of the martingale
$$
W_n''(0) = \frac 1 {m^n} \sum_{i=1}^{N_n} \left\{(z_{i,n} - \varphi'(0)n)^2 - \varphi''(0)n\right\}.
$$

%\begin{multline}
%\frac{L_n(k)}{W_\infty(0) m^{n}}
%=
%\frac{1}{\sqrt{2\pi n\varphi''(0)}}
%\exp\left\{-\frac {\left(k-\varphi'(0) n - \frac{W'_\infty(0)}{W_\infty(0)} %\right)^2}{2\varphi''(0)n}\right\} \\
%\cdot\left(1 + \frac{1}{\sqrt n}\frac{\kappa_3(0)}{6\sigma^3(0)}\Herm_3(x)\right) + o\left(\frac 1n\right)
%\quad \text{a.s.}
%\end{multline}

%$$
%\frac{L_n(k)}{W_\infty(0) m^{n}}
%=
%p\left(k; \varphi'(0) n + \frac{W'_\infty(0)}{W_\infty(0)}, \varphi''(0)n\right)
%\left(1 + \frac{1}{\sqrt n}\frac{\kappa_3(0)}{6\sigma^3(0)}\Herm_3(x)\right) + o\left(\frac 1n\right).
%$$

There is also a global central limit theorem for the BRW, originally known as the Harris conjecture~\cite[Chapter III, \S 16]{harris_book}.  It states that for all $x\in\R$,
\begin{equation}\label{eq:CLT_BRW}
\frac 1{m^{n}} \#\left\{1\leq i\leq N_n\colon \frac{z_{i,n}-\varphi'(0) n}{\sqrt{\varphi''(0) n}} \leq  x \right\}
\toas  \frac{W_\infty(0)}{\sqrt {2\pi}} \int_{-\infty}^{x} \eee^{-\frac12 y^2} \dd y.
%m^{-n} \sum_{k\leq \varphi'(0) n + x \sqrt{\varphi''(0) n}} L_n(k)
%\toas W_\infty(0) \Phi(x)
\end{equation}
Various forms of~\eqref{eq:CLT_BRW} and~\eqref{eq:local_limit_BRW} have been obtained in~\cite{stam,joffe_moncayo,asmussen_kaplan1,asmussen_kaplan2,biggins_CLT,uchiyama,biggins_uniform,yoshida,gao_liu_wang}.
%where $\Phi(x) = \frac{1}{\sqrt {2\pi}} \int_{-\infty}^{x} \eee^{-\frac12 t^2} \dd t$ is the standard normal distribution function.
%Various forms of~\eqref{eq:CLT_BRW} were established in~\cite{asmussen_kaplan1}, \cite{asmussen_kaplan2},  \cite{biggins_CLT}, \cite{uchiyama}, \cite{biggins_uniform}.
%\begin{figure}[!htbp]
%\includegraphics[width=0.49\textwidth]{BRW_LD.pdf}
%\caption
%{\small The  log-profile $k\mapsto \frac 1n \log L_n(k)$ of a branching random walk (blue), the information function $k\mapsto -I(\frac kn)$ (black), %and the approximation obtained by taking the logarithm in~\eqref{eq:Bahadur_Rao_BRW1} (red). }
%%$k\mapsto -\frac 1{2n} \log (2\pi \varphi''(\beta) n) \frac 1 n\log W_\infty(\frac kn)- I(\frac kn)$ . }
%\label{bild:BRW_LDP}
%\end{figure}
In fact, we can obtain a full asymptotic expansion in~\eqref{eq:CLT_BRW}. To this end, one takes sums in Theorem~\ref{theo:asympt_expansion_BRW} and uses the Euler--MacLaurin formula to approximate sums by integrals. We will record here only the first non-trivial term of this expansion.
\begin{proposition}\label{prop:CLT_global_speed}
Consider a branching random walk satisfying Assumptions A--E. Then,
\begin{multline}\label{eq:berry_esseen}
\frac 1 {W_\infty(0) m^n} \sum_{h=-\infty}^{k} L_n(h) - \frac {1}{\sqrt{2\pi}} \int_{-\infty}^{x_n(k)} \eee^{-\frac 12 z^2} \dd z =
\\
 \frac {\eee^{-\frac 12 x_n^2(k)}} {\sigma(0) \sqrt {2\pi n} }
\left(\frac 12 - \frac{\kappa_3(0)}{6\sigma^2(0)} (x_n^2(k) -1) - \frac {W_\infty'(0)}{W_\infty(0)}\right) + o\left(\frac 1 {\sqrt n}\right)
\quad\text{a.s.},
\end{multline}
where the $o$-term is uniform over $k\in\Z$.
\end{proposition}
A non-lattice version of Proposition~\ref{prop:CLT_global_speed} has been obtained recently in~\cite{gao_liu}. In fact, similar results hold for many models other than the BRW; see~\cite{kabluchko_distr_of_levels}.

%\begin{figure}[!htbp]
%\includegraphics[width=0.49\textwidth]{BRW_Berry_esseen.pdf}
%\caption
%{\small Blue: The left-hand side of~\eqref{eq:berry_esseen}. Red: Approximation term on the right-hand side of~\eqref{eq:berry_esseen}.}
%%$k\mapsto -\frac 1{2n} \log (2\pi \varphi''(\beta) n) \frac 1 n\log W_\infty(\frac kn)- I(\frac kn)$ . }
%\label{bild:BRW_Berry_Esseen}
%\end{figure}

\subsection{Comparing the profile and the expected profile}\label{subsec:expansion_expected}
It is interesting to compare the expansion of $L_n(k)$ stated in Theorem~\ref{theo:asympt_expansion_BRW} with the expansion of $\E L_n(k)$. Since the intensity measure $\vartheta_n$ of the branching random walk at time $n$ is just the $n$-fold convolution of $\vartheta_1$,  we can apply the classical expansion~\eqref{eq:asympt_expansion_random_walk} to the expected profile $\E L_n(k) = \vartheta_n(\{k\})$. The proof of the following proposition is standard and will be given in Section~\ref{sec:proof_exp_char}. Formally, it can be obtained from Theorem~\ref{theo:asympt_expansion_BRW} by taking the expectation and noting that $\E W_\infty(\beta)=1$  for all $\beta\in (\beta_-,\beta_+)$ which implies that all higher derivatives of $W_\infty(\beta)$ have zero expectation. %$\E W_{\infty}^{(i)}(\beta)=0$ for all $i\in\N$ and $\beta\in (\beta_-,\beta_+)$.

\begin{proposition}\label{prop:expansion_expected}
Consider a branching random walk satisfying Assumptions A, C, E. Fix $r\in\N_0$ and a compact interval $K\subset
\cD_\varphi$.
Then we have
\begin{equation}\label{eq:asympt_expansion_intensity}
\eee^{\beta  k - \varphi(\beta )n} \E L_n(k)
=
\frac {\eee^{-\frac 12 x^2_n(k)}} {\sigma(\beta) \sqrt {2\pi n}}
\sum_{j=0}^{r}
\frac 1 {n^{j/2}}q_j(x_n(k);\beta) + o\left(n^{-\frac{r+1}2}\right),
\end{equation}
where $q_j(x;\beta)$ is the same as $q_j(x)$ in~\eqref{eq:asympt_expansion_random_walk} but with $\kappa_j$ replaced by $\kappa_{j}(\beta)$. The $o$-term is uniform in $\beta \in K$ and $k\in\Z$.
\end{proposition}

We can considerably simplify the expansion given in Theorem~\ref{theo:asympt_expansion_BRW} if instead of $L_n(k)$ we consider $L_n(k) - W_\infty(\beta ) \E L_n(k)$. This normalization removes all terms involving $W_\infty(\beta )$ and we obtain
\begin{corollary}\label{cor:asympt_expansion_BRW_without_expectation}
Consider a branching random walk satisfying Assumptions A--E. Fix $r\in\N_0$ and a compact interval $K\subset (\beta_-,\beta_+)$. Then,
\begin{multline}
\eee^{\beta k - \varphi(\beta )n} (L_n(k) - W_\infty(\beta )\, \E L_n(k))\\
=\frac {\eee^{-\frac 12 x^2_n(k)}} {\sigma(\beta) \sqrt {2\pi n}}
\sum_{j=1}^{r}
\frac 1 {n^{j/2}}
F_j^{\circ}(x_n(k);\beta ) +  o\left(n^{-\frac{r+1}2}\right)
\quad \text{a.s.},\label{eq:exp_brw_simplified}
\end{multline}
where $F^{\circ}_j(x;\beta )$  is the same as $F_j(x;\beta )$ but without the term involving $W_\infty(\beta )$, that is
\begin{equation}
F^{\circ}_j(x;\beta ):= F_j(x;\beta )-W_\infty(\beta )q_j(x;\beta ).
\end{equation}
The $o$-term in~\eqref{eq:exp_brw_simplified} is uniform in $\beta \in K$ and $k\in\Z$.
\end{corollary}
Note that the summation in~\eqref{eq:exp_brw_simplified} starts with $j=1$.   We will need the first two terms of the expansion, c.f.~\eqref{eq:F_1_BRW} and~\eqref{eq:F_2_BRW}:
\begin{align}
F_1^\circ(x;\beta ) &=  \frac{x}{\sigma(\beta)} W_\infty' (\beta ),\label{eq:F_1_circ_BRW}\\
F_2^\circ(x;\beta ) &=  \frac{\kappa_3(\beta )}{6\sigma^4(\beta)}W_\infty'(\beta )\Herm_4(x)
+  \frac{1}{2\sigma^2(\beta)}W_\infty''(\beta )\Herm_2(x).\label{eq:F_2_circ_BRW}
\end{align}

%\begin{theorem}\label{theo:asympt_exp_BRW_3terms_expect}
%Fix $r\in\N_0$. Under the above assumptions,  we have
%\begin{align*}
%n^{(r+1)/2}
%\sup_{k\in\Z}\left|\eee^{\beta k - \varphi(\beta )n} (L_n(k) - W_\infty(\beta )\E L_n(k)) -
%\right| \toas 0
%,
%\end{align*}
%where $F_{j}^{\circ}(x;\beta )$ is a degree $3r$ polynomial in $x=x_n(k)$ whose random coefficients can be expressed %through the cumulants $\kappa_1(\beta ),\ldots, \kappa_{j+1}(\beta )$ and the derivatives of $W_\infty$ at $\beta $ of %orders $1,\ldots,j$.
%%The expansion holds uniformly over $k\in\Z$.
%%satisfying $x\in (-\eps,\eps)$.
%\end{theorem}

\subsection{Uniform expansions}
For every point $k\in\Z$ Theorem~\ref{theo:asympt_expansion_BRW} yields a family of asymptotic expansions of $L_n(k)$ parametrized by $\beta$.
It is natural to choose $\beta=\beta(\frac kn)$ as the solution to $\varphi'(\beta(\frac kn))=\frac kn$ because then $x_n(k)=0$. The \textit{information function} $I$ is defined by
$$
I(\theta) = \beta \varphi'(\beta) - \varphi(\beta) \text{ for } \theta = \varphi'(\beta).
$$
With the above choice of $\beta$, Theorem~\ref{theo:asympt_expansion_BRW} and Proposition~\ref{prop:expansion_expected} yield the following result.
\begin{corollary}\label{cor:asympt_expansion_BRW_optimal}
Consider a branching random walk satisfying Assumptions A--E. Fix $r\in\N_0$. Then we have
\begin{equation}\label{eq:L_n_uniform_exp}
%n^{\frac {r+1}2}
%\sup_{k\in\Z\colon  \frac kn \in K}\left|
\eee^{nI(\frac kn)} L_n(k)
=
\frac 1 {\sigma(\beta(\frac kn)) \sqrt {2\pi n}}
\sum_{j=0}^{r}
\frac 1 {n^{j/2}}F_j\Bigl(0;\beta\Bigl(\frac kn\Bigr)\Bigr)
+o\left( n^{-\frac {r+1}2}\right) \quad\text{a.s.},
%\right| \toas 0
\end{equation}
where the $o$-term is uniform over all $k\in\Z$ for which $\beta(\frac kn)$ exists and stays in a fixed compact set $K\subset (\beta_-,\beta_+)$. Similarly,
%Fix $r\in\N_0$. Let $K$ be a compact subset of the interval $\cD_\varphi$. Then,
\begin{equation}\label{eq:EL_n_uniform_exp}
%\sup_{k\in\Z\colon  \frac kn \in K}
%\left|
\eee^{nI(\frac kn)} \E L_n(k)
=
\frac 1 {\sigma(\beta(\frac kn)) \sqrt {2\pi n}}
\sum_{j=0}^{r}
\frac 1 {n^{j/2}}q_j\Bigl(0;\beta\Bigl(\frac kn\Bigl)\Bigr)
+ o\left( n^{-\frac {r+1}2}\right),
%\right| = 0.
\end{equation}
where the $o$-term is uniform over all $k\in\Z$ such that $\beta(\frac kn)$ exists and stays in a fixed compact set $K\subset \cD_\varphi$.
\end{corollary}
The expansion of $\E L_n(k)$ is well known because it is a classical formula for the precise large deviations of sums of i.i.d.\ random variables; see~\cite{blackwell_hodges,bahadur_rao,petrov_ld}. Note that the terms with odd $j$ in~\eqref{eq:L_n_uniform_exp} and~\eqref{eq:EL_n_uniform_exp} vanish; see Remark~\ref{rem:odd} below.
\begin{remark}
By taking a linear combination of~\eqref{eq:L_n_uniform_exp} and~\eqref{eq:EL_n_uniform_exp} we obtain an expansion of
$$
\eee^{nI(\frac kn)} \left(L_n(k)- W_\infty\Bigl(\beta\Bigl(\frac kn\Bigr)\Bigr) \E L_n(k)\right)
$$
which looks exactly as~\eqref{eq:L_n_uniform_exp} but with $F_j$ replaced by $F_j^{\circ}$.
\end{remark}

\begin{figure}[!htbp]
\includegraphics[width=0.49\textwidth]{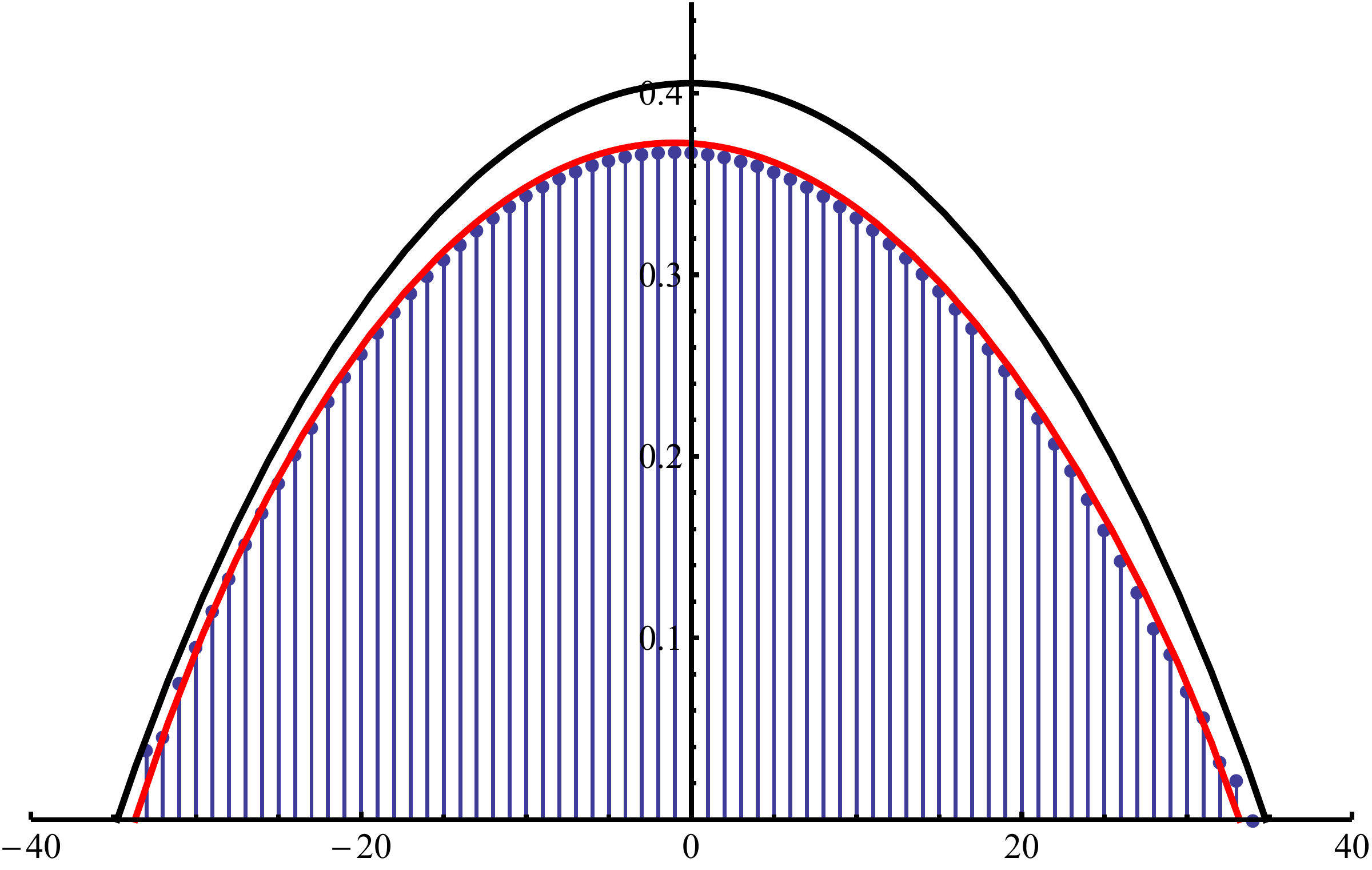}
\caption
{\small A realization of the log-profile $k\mapsto \frac 1n \log L_n(k)$ of a branching random walk (blue), the information function $k\mapsto -I(\frac kn)$ (black), and the approximation obtained by taking the logarithm in~\eqref{eq:Bahadur_Rao_BRW1} (red) at time $n=50$. The realization is the same as in Figure~\ref{bild:BRW_CLT}.}
%$k\mapsto -\frac 1{2n} \log (2\pi \varphi''(\beta) n) \frac 1 n\log W_\infty(\frac kn)- I(\frac kn)$ . }
\label{bild:BRW_LDP}
\end{figure}

Taking $r=0$ in the above theorem yields the following known result~\cite{biggins_growth_rates,biggins_uniform}; see Figure~\ref{bild:BRW_LDP}. In the context of random trees, results similar to Corollary~\ref{cor:L_n_divided_by_EL_n} have been obtained by \citet{chauvin_drmota_jabbour}, \citet{drmota_janson_neininger}, \citet[Theorem~3.1]{chauvin_etal} and~\citet{sulzbach}.
\begin{corollary}\label{cor:L_n_divided_by_EL_n}
Let $K$ be a compact subset of the interval $(\beta_-,\beta_+)$ (for~\eqref{eq:Bahadur_Rao_BRW1}), resp.\ $\cD_\varphi$ (for~\eqref{eq:Bahadur_Rao_BRW2}). Then,
\begin{align}
&\sup_{k\in\Z\colon \beta(\frac kn) \in K} \left|  \sqrt {2\pi n}\, \sigma\Bigl(\beta\Bigl(\frac kn\Bigr)\Bigr)\, \eee^{nI(\frac kn)}\, L_n(k) - W_\infty\Bigl(\beta\Bigl(\frac kn\Bigr)\Bigr) \right| \toas 0, \label{eq:Bahadur_Rao_BRW1}\\
&\sup_{k\in\Z\colon \beta(\frac kn) \in K} \left|  \sqrt {2\pi n}\, \sigma\Bigl(\beta\Bigl(\frac kn\Bigr)\Bigr)\, \eee^{nI(\frac kn)}\, \E L_n(k) - 1\right| \ton 0. \label{eq:Bahadur_Rao_BRW2}
%\sup_{k\in\Z\colon \beta(\frac kn) \in K} \left| \frac{L_n(k)}{\E L_n(k)} - W_\infty\left(\beta\left(\frac kn\right)\right) \right| \toas 0.
\end{align}
\end{corollary}
If $k=k_n$ is an integer sequence such that $\lim_{n\to\infty} k_n/n = \varphi'(\beta)$ with some $\beta\in (\beta_-,\beta_+)$ (for~\eqref{eq:Bahadur_Rao01}) or $\beta\in\cD_\varphi$ (for~\eqref{eq:Bahadur_Rao02}), then $\beta(\frac {k_n}n) \to \beta$ and we obtain that
\begin{align}
&\sqrt{2\pi \varphi''(\beta) n} \, \eee^{n I(\frac{k_n}{n})} L_n(k_n) \toas  W_\infty (\beta), \label{eq:Bahadur_Rao01}
\\
&\sqrt{2\pi \varphi''(\beta) n} \, \eee^{n I(\frac{k_n}{n})} \E L_n(k_n) \ton  1. \label{eq:Bahadur_Rao02}
\end{align}
%where for $\theta=\varphi'(\beta)$,  the \textit{information function} $I$ is defined by  $I(\theta) =\varphi'(\beta)\beta - \varphi(\beta)$. As a %consequence, one has
%\begin{equation}\label{eq:Bahadur_Rao_BRW}
%%\sup_{k\in\Z\colon \beta(\frac kn) \in K}
%\frac{L_n(k_n)}{\E L_n(k_n)} \toas  W_\infty (\beta).
%\end{equation}
%Finally, the third fundamental result describes the shape of the branching random walk in the domain of large deviations: if $k=k_n$ is a sequence such %that $\lim_{n\to\infty} k_n/n = \varphi'(\beta)$ for some $\beta\in (\beta_-,\beta_+)$, then
%\begin{align}
%&\sqrt{2\pi \varphi''(\beta) n} \, \eee^{n I(\frac{k_n}{n})} L_n(k_n) \toas  W_\infty (\beta), \label{eq:Bahadur_Rao_BRW1}\\
%&\sqrt{2\pi \varphi''(\beta) n} \, \eee^{n I(\frac{k_n}{n})} \E L_n(k_n) \toas  1, \label{eq:Bahadur_Rao_BRW2}
%\end{align}
%where for $\theta=\varphi'(\beta)$,  the \textit{information function} $I$ is defined by  $I(\theta) =\varphi'(\beta)\beta - \varphi(\beta)$. As a %consequence, one has
%\begin{equation}\label{eq:Bahadur_Rao_BRW}
%%\sup_{k\in\Z\colon \beta(\frac kn) \in K}
%\frac{L_n(k_n)}{\E L_n(k_n)} \toas  W_\infty (\beta).
%\end{equation}
%We will derive complete asymptotic expansions in~\eqref{eq:CLT_BRW}, \eqref{eq:Bahadur_Rao_BRW1}, \eqref{eq:Bahadur_Rao_BRW2}, %\eqref{eq:Bahadur_Rao_BRW} .

Equation~\eqref{eq:L_n_uniform_exp} of Corollary~\ref{cor:asympt_expansion_BRW_optimal} holds uniformly in $k\in\Z$ as long as $\frac kn$ stays in the interval  $(\varphi'(\beta_-), \varphi'(\beta_+))$ and remains bounded away from its ends.  Here, both derivatives are understood as the corresponding one-sided limits and are allowed to be infinite. One may ask whether the whole range of the BRW is covered or just some part of it. A basic result due to~\citet{biggins_chernov}, see also~\cite{biggins_branching_out}, describes the asymptotic range of the branching random walk:
\begin{equation}\label{eq:chernoff_BRW}
\frac 1n \max_{i=1,\ldots, N_n} z_{i,n} \toas \Gamma_+,
\quad
\frac 1n \min_{i=1,\ldots, N_n} z_{i,n} \toas \Gamma_-,
\end{equation}
where
\begin{equation}\label{eq:Gamma_+_Gamma_-}
\Gamma_+ =  \inf_{\beta>0} \frac{\varphi(\beta)}{\beta}>\varphi'(0),
\quad
\Gamma_- =  \sup_{\beta<0} \frac{\varphi(\beta)}{\beta}<\varphi'(0).
\end{equation}
If we additionally assume that $\varphi$ is finite everywhere on $\R$ (rather than on some interval) or that the support of the intensity measure $\vartheta_1$ is finite, then it is easy to prove (see, e.g.,\ \cite{petrov_ld}) that $\varphi'(\beta_-)=\Gamma_-$ and $\varphi'(\beta_+)=\Gamma_+$. In this case, Equation~\eqref{eq:L_n_uniform_exp} of Corollary~\ref{cor:asympt_expansion_BRW_optimal} covers the whole range of the ``central order statistics'' of the BRW. The ``extremal'' and ``intermediate'' order statistics at $\Gamma_\pm n \mp o(n)$ are not covered.  However, in some (rather exotic) examples it is possible that $(\varphi'(\beta_-), \varphi'(\beta_+))$ is a strict subinterval of $(\Gamma_-, \Gamma_+)$.

\begin{example}[see~\cite{petrov_ld}]
Consider a BRW for which $\vartheta_1(\{k\}) = c \eee^{-k}/k^3$ for $k\in\N$ and $\vartheta_1(\{k\})=0$ otherwise. Here, $c>0$ is a parameter. Then $\varphi(\beta)$ is finite for $\beta\leq 1$ and moreover, the left derivative $\varphi'(1-):=\lim_{\beta\uparrow 1}\varphi'(\beta)$ is also finite, whereas for $\beta>1$ one has $\varphi(\beta)=+\infty$.
\begin{center}
\includegraphics[width=0.45\textwidth]{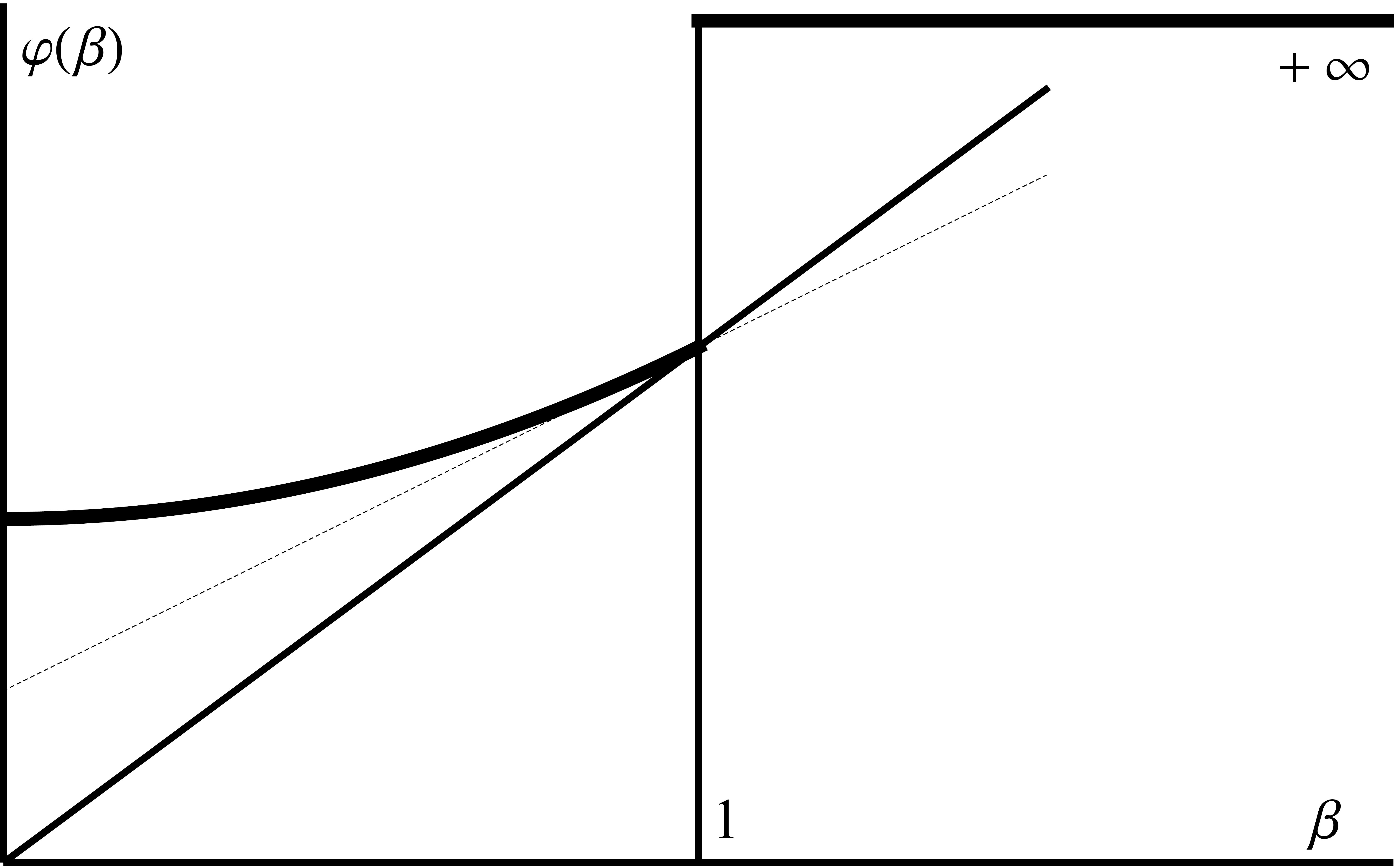}
\end{center}
If $c$ is sufficiently large, then the BRW is supercritical and we have $\varphi'(1)<\varphi(1)$. This means that $\beta_+=1$, whereas  $\Gamma_+=\varphi(1)$. Hence $\Gamma_+$ is strictly larger than $\varphi'(\beta_+)$.
\end{example}

\subsection{Continuous-time branching random walks}\label{subsec:cont_time_BRW}
All results of this paper apply to \textit{continuous-time} branching random walks on $\Z$ which are defined as follows. At time $0$ one particle appears at position $0$. After an exponential time with parameter $\lambda>0$, the particle disappears and at the same moment of time it is replaced by a random cluster of particles whose displacements w.r.t.\ the original particle are distributed according to some fixed point process $\zeta$ on $\Z$. The new-born particles behave in the same way as the original particle. All the random mechanisms involved are assumed to be independent.  Denote the number of particles at time $t\geq 0$ by $N_t$ and  note that $\{N_t\colon t\geq 0\}$ is a branching (Markov) process in continuous time; see~\cite[Chapter~V]{harris_book}. Note that the law of the continuous-time BRW is uniquely determined by the following two parameters: the intensity $\lambda$ and the law of the point process $\zeta$.
%Let $z_{1,t} \leq \ldots \leq z_{N_t,t}$ be the positions of the particles at time $t$.
The occupation number $L_t(k)$ is defined as the number of particles located at site $k\in\Z$ at time $t\geq 0$. If we restrict the time $t$ to integer values only, we obtain a discrete-time BRW called the ``discrete skeleton'' of the original continuous-time BRW. If Assumptions A--E of Section~\ref{subsec:BRW_def_assumpt} hold for the discrete skeleton, then all the results of the present paper can be translated in an evident way to the continuous-time setting by replacing $n\in\N_0$ by $t\geq 0$. Note, however, that one has to be careful whenever the arithmetic properties of $\varphi'(0)$ are involved; see Sections~\ref{subsec:xia_chen}, \ref{subsec:width_mode}.  The proofs require only straightforward modifications.

\subsection{Strong limit theorems for the occupation numbers \texorpdfstring{$L_n(k_n)$}{Ln(kn)}}\label{subsec:xia_chen}
Recall that $L_n(k)$ denotes the number of particles located at time $n$ at $k\in\Z$. Let us take an integer sequence $k=k_n$ which behaves in some regular way.  We ask whether the random variables $L_n(k_n)$ have a non-degenerate a.s.\ limit, after an appropriate affine normalization. The next proposition is known and follows immediately from~\eqref{eq:Bahadur_Rao01} and~\eqref{eq:Bahadur_Rao02}. %Corollary~\ref{cor:L_n_divided_by_EL_n}.
\begin{proposition}\label{prop:L_n_divided_EL_n}
Consider a branching random walk satisfying Assumptions A--E. Let $k_n$ be an integer sequence such that $k_n \sim \varphi'(\beta )n$ for some $\beta \in (\beta_-,\beta_+)$. Then we have
%$$
%k_n = \mu  n +  o(n)\in\Z, \quad n\to\infty.
%$$
%Then, we have
\begin{equation}\label{eq:L_n_lim_distr_1}
%\sigma \sqrt n \,\eee^{\beta k_n - \varphi(\beta )n} L_n(k_n)
\frac{L_n(k_n)}{\E L_n(k_n)}
\toas
W_{\infty} (\beta ).
\end{equation}
\end{proposition}
%If the random variable $W_\infty(\beta )$ is non-degenerate, this gives the desired limiting distribution for $L_n(k_n)$. However, it is possible that $W_\infty (\beta )=1$ a.s. Note that $1$ is the only possible value because $\E W_\infty(\beta )=1$.  The most interesting case in which $W_\infty(\beta )=1$ is when $\beta =0$ and any particle in the BRW generates a deterministic number $m\in\N$ of descendants. In this case, we ask whether there is a better affine normalization of $L_n(k_n)$ leading to a non-degenerate limit law.
Next we ask whether we can obtain more refined limit theorems for the ``centered'' variables
\begin{equation}
L_n^{\circ}(k_n) := L_n(k_n) - W_\infty(\beta)\, \E L_n(k_n).
\end{equation}
This question is especially natural if $\beta=0$ and any particle in the BRW generates the same number $m\in\{2,3,\ldots\}$ of descendants. Then $W_\infty(0)=1$ and hence the limit random variable provided by Proposition~\ref{prop:L_n_divided_EL_n} is a.s.\ constant. The same phenomenon occurs in the setting of random trees, where the natural analogue of $W_\infty(0)$ is equal to $1$.

We consider an integer sequence $k_n$ which, for some $\beta\in (\beta_-,\beta_+)$, is represented in the form
$k_n=\varphi'(\beta) n + c_n$. The result will depend on the asymptotic behavior of $c_n$.
%We consider the cases
%where $c_n$ is asymptotically a multiple of
%$\sqrt{n}$, where $c_n$ tends to $\infty$ but slower than $\sqrt{n}$, and where $c_n$ is bounded.
Recall that $\sigma^2(\beta) = \varphi''(\beta)$.

\begin{theorem}\label{theo:L_n_lim_distr_2_new}
Consider a branching random walk satisfying Assumptions A--E.
Let $k_n$ be an integer sequence such that
\begin{equation*}
   k_n= \varphi'(\beta)\, n +c_n
\end{equation*}
for some $\beta\in (\beta_-,\beta_+)$ and some $c_n=O(\sqrt{n})$.

\vspace*{1mm}

\emph{(a)} If $c_n= \alpha \sigma(\beta) \sqrt {n} + o(\sqrt n)$ for some $\alpha\in\bR$, then
\begin{equation}\label{eq:L_n_lim_distr_2}
      n \eee^{\beta k_n - \varphi(\beta)n}  L_n^{\circ}(k_n) \toas
         \frac{1}{\sqrt{2\pi}\sigma^2(\beta)} \alpha \eee^{-\frac {1}{2}\alpha^2 } W_\infty'(\beta).
\end{equation}

\emph{(b)} If $\,\lim_{n\to\infty} c_n = +\infty$, but  $c_n=o(\sqrt n)$, then
\begin{equation}\label{eq:xia_chen_general1}
        n^{3/2} c_n^{-1} \eee^{\beta k_n - \varphi(\beta)n} L_n^{\circ}(k_n) \toas
                                    \frac {1}{\sqrt{2\pi} \sigma^3(\beta)}W_\infty'(\beta).
\end{equation}

\emph{(c)} If $c_n$ is bounded, then
\begin{equation}\label{eq:xia_chen_general}
       n^{3/2} \eee^{\beta k_n-\varphi(\beta) n} L_n^{\circ}(k_n) - R(c_n) \toas 0
\end{equation}
where
\begin{equation}
    R(c) := \frac 1 {\sqrt{2\pi}\sigma^3(\beta)} \left(W_\infty'(\beta) \left( c
      + \frac{\kappa_3(\beta)}{2\sigma^2(\beta)}\right) - \frac 12 W_\infty''(\beta)\right).
\end{equation}
\end{theorem}

In the setting of random trees, \citet{fuchs_hwang_neininger} showed that the analogue of the
occupation number $L_n(k_n)$, centered by its expectation and normalized by standard deviation,
does not have a limit distribution $k_n=\varphi'(\beta) n + O(1)$. Part (c) of the above theorem clarifies
the structure of the set of a.s.\ subsequential limits in the BRW setting.

\begin{figure}
\includegraphics[width=0.99\textwidth]{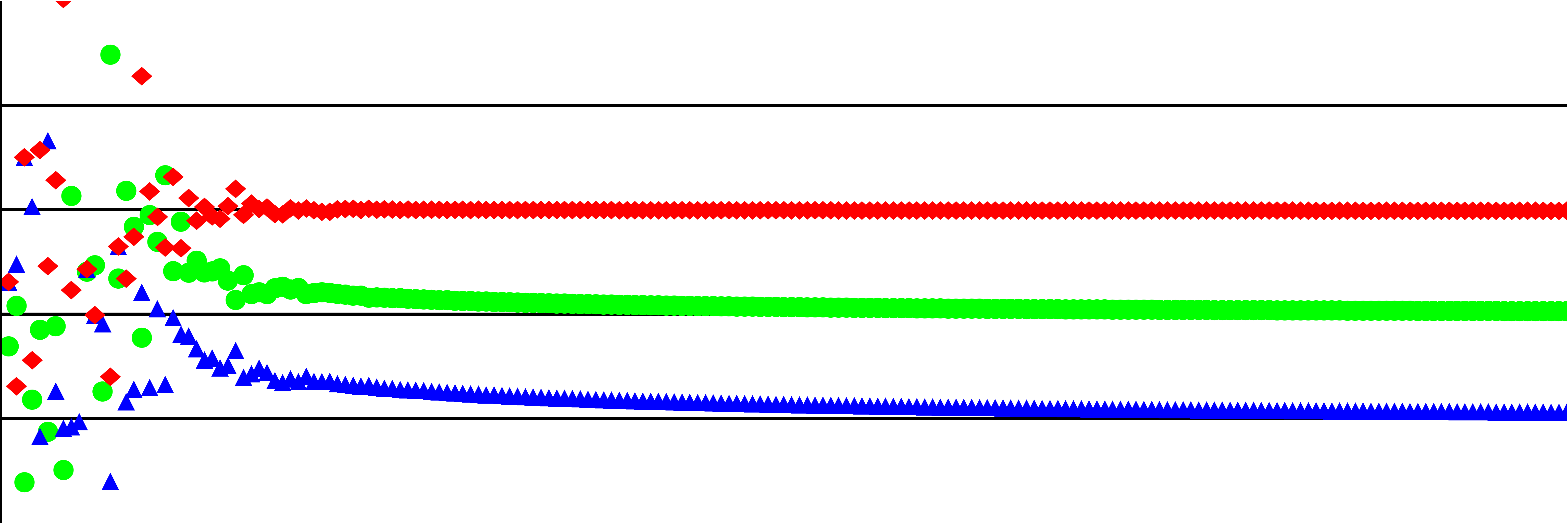}
\includegraphics[width=0.99\textwidth]{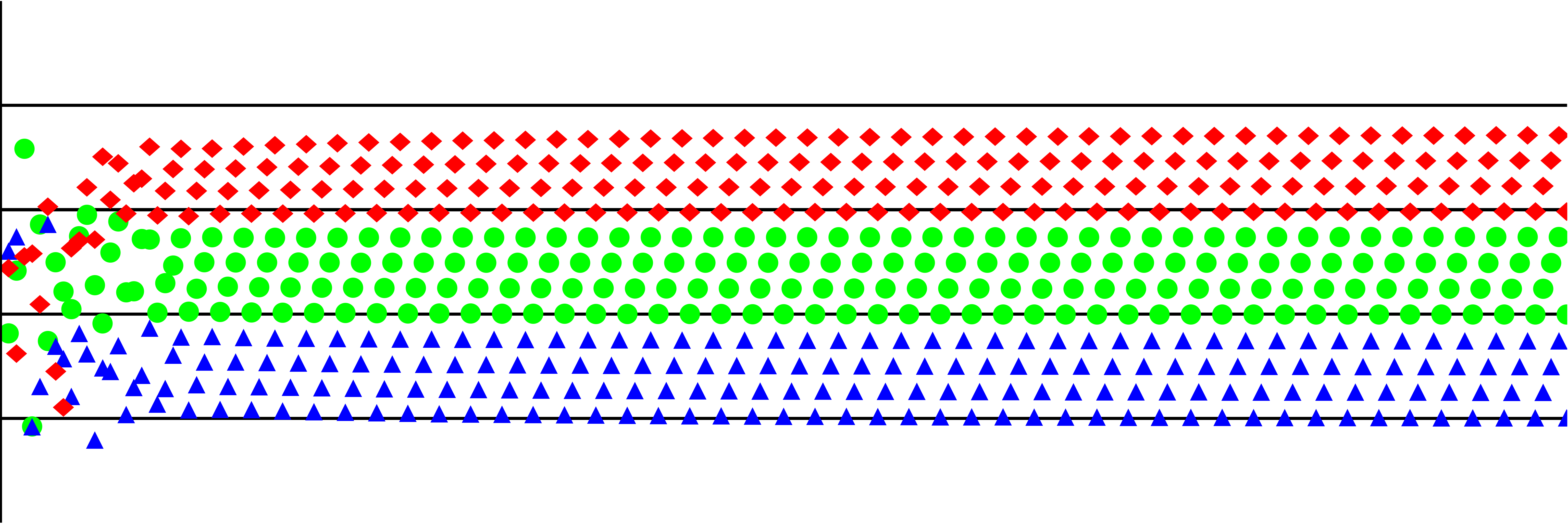}
\includegraphics[width=0.99\textwidth]{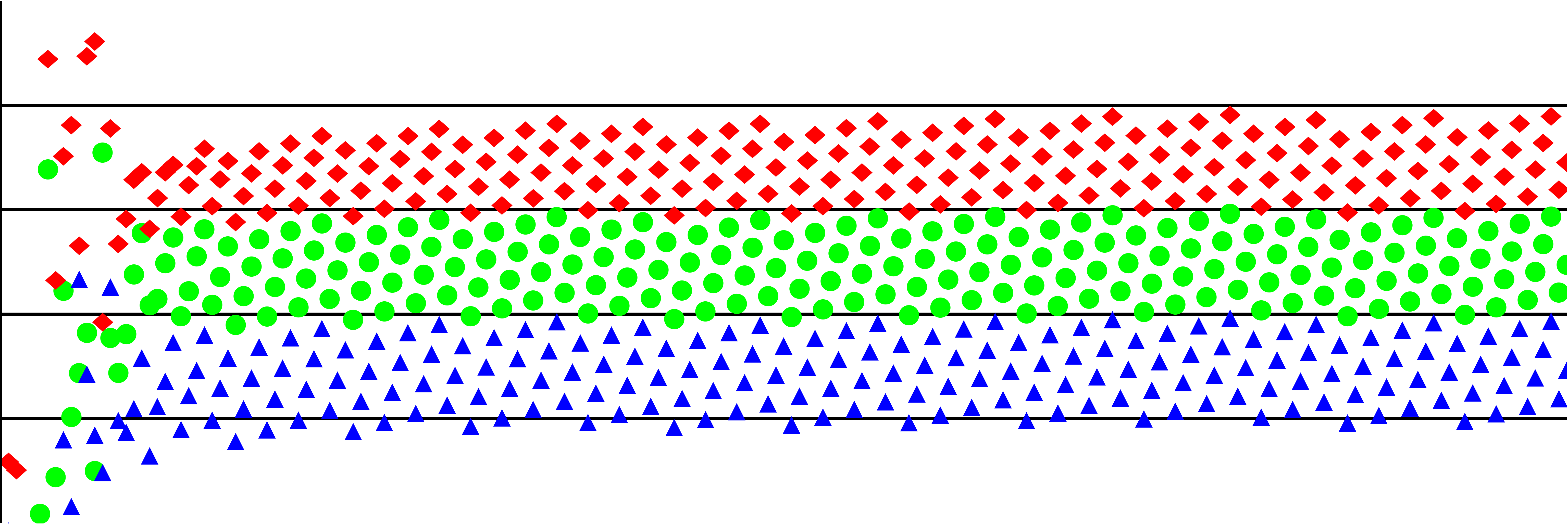}
\caption
{\small A realization of the sequence $n^{3/2} m^{-n} L_n^{\circ}(k_n)$ as a function of $n=1,\ldots,200$ for $k_n=\lfloor\varphi'(0)n\rfloor + a$ with $a=-1,0,1$ (red diamonds, green circles, blue triangles, respectively). The four black horizontal lines show the values $R(-2),R(-1), R(0), R(1)$. Top: $\varphi'(0)$ is integer, middle: $\varphi'(0)=\frac 14$ is rational, bottom: $\varphi'(0)$ is irrational.}
\label{bild:BRW_L_n_k_n}
\end{figure}

\begin{remark}
In the situation of Theorem~\ref{theo:L_n_lim_distr_2_new}\,(c), $\limsup$ and $\liminf$ of the sequence
\begin{equation}\label{eq:xia_chen_seq}
 n^{3/2}\eee^{\beta k_n - \varphi(\beta)n} L_n^{\circ}(k_n)
\end{equation}
are a.s.\ finite (but not necessarily equal to each other).
Whether or not the sequence~\eqref{eq:xia_chen_seq} has an a.s.\ limit depends on certain arithmetic issues;
see Figure~\ref{bild:BRW_L_n_k_n}. We consider the following cases.

\vspace*{1mm}
\noindent
\textit{Case 1.} If the ``drift'' $\varphi'(\beta)$ is integer, it is natural to take $k_n=\varphi'(\beta)n + a$ for some constant $a\in\Z$. In this case, $c_n = a$ which implies that $R(c_n)$ converges a.s.\ to $R(a)$ and so does the sequence~\eqref{eq:xia_chen_seq}. This case is shown on top of Figure~\ref{bild:BRW_L_n_k_n}.

\vspace*{1mm}
\noindent
\textit{Cases 2 and 3.}
If $\varphi'(\beta)$ is non-integer we cannot choose $k_n$ as in Case 1. Instead, it is natural to take $k_n=\lfloor\varphi'(\beta)n\rfloor + a$, where $a\in\Z$, which means that $c_n = a - \{\varphi'(\beta) n\}$. Here, $\lfloor\cdot\rfloor$  denotes the floor function and $\{\cdot\}$ denotes the fractional part. The limit of $c_n$ (and hence the limit of $R(c_n)$) does not exist. Instead, we can  describe the set of all a.s.\  subsequential limits of~\eqref{eq:xia_chen_seq}.

\vspace*{1mm}
\noindent
\textit{Case 2.} For non-integer rational  $\varphi'(\beta)=\frac{p}{q}$ with coprime $p$ and $q$ the set of a.s.\ subsequential limits of~\eqref{eq:xia_chen_seq} is finite and given by
\begin{equation}\label{eq:limiting_distr1}
\left\{R\left(a-\frac jq\right) \colon j=0,\ldots, q-1\right\}.
\end{equation}
In fact, over the subsequence $n_l = j + l q$, $l\in\N_0$, we have
$$
n^{3/2}\eee^{\beta k_n - \varphi(\beta)n} L_n^{\circ}(k_n) \toas R\left( a - \left\{\frac{pj}{q}\right\}\right).
$$
This case (with $q=4$) is shown in the middle of Figure~\ref{bild:BRW_L_n_k_n}.

\vspace*{1mm}
\noindent
\textit{Case 3.}
For irrational $\varphi'(\beta)$ the set of a.s.\ subsequential limits of~\eqref{eq:xia_chen_seq} can be continuously parametrized by the interval $[0,1]$:
\begin{equation}\label{eq:limiting_distr2}
%\left\{ W_\infty'(\beta) \left( a- z  + \frac{\kappa_3(\beta)}{2\sigma^2(\beta)}\right) - \frac 12 W_\infty''(\beta)\colon z\in [0,1]\right\}.
\{R\left(a-z\right) \colon z\in [0,1] \}.
\end{equation}
This case is shown on the bottom of Figure~\ref{bild:BRW_L_n_k_n}.
\end{remark}

\begin{remark}
The above Cases~1 and~2 apply to BRW in discrete time only. In the case of continuous time $t\geq 0$, see Section~\ref{subsec:cont_time_BRW}, it is natural to take $k_t =\lfloor \varphi'(\beta) t \rfloor + a$ with $a\in\Z$. We have $c_t = a - \{\varphi'(\beta) t\}$ and there are two cases.

\vspace*{1mm}
\noindent
\textit{Case 1.} If $\varphi'(\beta) = 0$, then $c_t=a$ and the sequence~\eqref{eq:xia_chen_seq} has an a.s.\ limit $R(a)$.

\vspace*{1mm}
\noindent
\textit{Case 2.} If $\varphi'(\beta)\neq 0$, then the set of limit points of $c_t$ is the interval $[a-1,a]$ and the set of a.s.\ subsequential limits of~\eqref{eq:xia_chen_seq} is given by~\eqref{eq:limiting_distr2}.
\end{remark}

\begin{example}\label{ex:xia_chen}
%In the case of the \textit{simple symmetric} branching random walk .
Consider a \textit{simple symmetric} branching random walk on $\Z$. That is, if at time $n$ some particle is located at
$k\in\Z$, then at time $n+1$ it moves to one of the sites $k-1$ or $k+1$ with probability $1/2$ and generates there a
random number of offspring. The distribution of the number of offspring is assumed to have finite mean $m>1$ and finite
$p$-th moment, for some $p>1$. Note that $\varphi'(0)=0$ by symmetry. Denote by $z_{1,n}, \ldots,z_{n,N_n}$ the
positions of particles at time $n$.  Note that although Assumption~E fails to hold for the simple symmetric BRW because
the intensity measure $\vartheta_1$ is concentrated on $\{-1,1\}\subset 2\Z+1$, the transformed BRW $\frac
12(z_{i,n}+n)$, $1\leq i\leq N_n$, satisfies Assumptions A--E. Applying Theorem~\ref{theo:mode_height}\,(c) to the new
BRW and switching back to the simple symmetric BRW, we recover a result obtained by~\citet[Theorem~2.1]{chen} under a
second moment condition on $N_1$. Parts (a) and (b) of Theorem~\ref{theo:mode_height} describe further possible
asymptotic regimes that are not covered by~\cite{chen}. Our Proposition~\ref{prop:CLT_global_speed} generalizes
Theorem~2.2 of~\cite{chen} to more general BRW's.
Very recently, \citet{gao_liu} elaborated on the result of~\citet{chen} by relaxing his second moment condition on $N_1$, allowing for general displacements distributions and by considering branching random walks in random environment. Note that~\citet{gao_liu} impose a strong non-lattice assumption on their BRW's and so their results do not cover the lattice case considered here. In particular, \citet{gao_liu} obtained  non-lattice versions of Theorem~\ref{theo:mode_height}\,(c) and Proposition~\ref{prop:CLT_global_speed}.
\end{example}

\subsection{Height and mode of the branching random walk}\label{subsec:width_mode}
Define the \textit{height} $M_n$ and the \textit{mode} $u_n$ of the branching random walk profile at time $n$ by
\begin{equation}\label{eq:width_mode_def}
M_n = \max_{k\in\Z} L_n(k),
\quad
u_n =\argmax_{k\in\Z} L_n(k).
\end{equation}
In the context of profiles of random trees, similar quantities (called width and mode) were studied by~\citet{chauvin_drmota_jabbour}, \citet{katona}, \citet{drmota_hwang} and~\citet{devroye_hwang}. Here, we will consider the more general setting of branching random walks; applications to random trees are postponed to a separate paper. It is clear from the approximate Gaussianity of the profile, c.f.~\eqref{eq:local_limit_BRW}, that the mode should be near $\varphi'(0)n$, while the height should be approximately $m^n/(\sqrt{2\pi n} \, \sigma(0))$.

\begin{theorem}\label{theo:mode_height}
Consider a branching random walk satisfying Assumptions A--E.

\vspace*{1mm}

\emph{(a)} There is an a.s.\ finite random variable $N$ such that for $n>N$,
the mode $u_n$ is equal to one of the numbers $\lfloor u_n^*\rfloor$ or $\lceil u_n^*\rceil$, where
\begin{equation}\label{eq:k_n_star}
    u_n^* = \varphi'(0) n + \frac{W_\infty'(0)}{W_\infty(0)} - \frac{\kappa_3(0)}{2\sigma^2(0)}
\end{equation}
and $\lfloor\cdot\rfloor$, $\lceil\cdot\rceil$ denote the floor and the ceiling functions, respectively.

\vspace*{.5mm}

\emph{(b)} The height $M_n$ satisfies
$$
\frac{\sqrt{2\pi n}\, \sigma(0)\, M_n}{W_\infty(0) m^n} = 1 - \frac{1}{2\sigma^2(0) n}
\left( \frac{\kappa_3^2(0)}{6\sigma^4(0)} - \frac{\kappa_4(0)}{4\sigma^2(0)} + \theta_n^2
+(\log W_\infty)''(0)
\right) + o\left(\frac 1n \right),
$$
a.s.,\ where $\theta_n := \min_{k\in\Z} |u_n^*-k|$ is the distance between $u_n^*$ and the closest integer.
\end{theorem}

The second part of the above theorem obviously implies that
\begin{equation}\label{eq:Max_weak}
   \frac{\sqrt{2\pi n} \, \sigma(0)\, M_n} {m^{n}}  \toas W_\infty(0).
\end{equation}
In the context of random trees, similar results were obtained
in~\cite{chauvin_drmota_jabbour,katona,drmota_hwang,devroye_hwang}. It has been
asked by~\citet[Section~5]{drmota_hwang} whether $M_n$,
after appropriate centering and scaling, converges in distribution,
and whether the limit distribution is (in our notation) the distribution of the logarithmic derivative
$W_\infty'(0)/W_\infty(0)$. Note that the analogue of $W_\infty(0)$ equals $1$ in the setting
of random trees, so that~\eqref{eq:Max_weak} gives an a.s.\ constant, degenerate limit.
Theorem~\ref{theo:mode_height}\,(b) answers these questions in the context of the BRW. In fact, we even have a
limit theorem on the a.s.\ behavior of $M_n$. In our result, the second (rather than the first)
logarithmic derivative of $W_\infty$ at $0$ appears.

\begin{remark} The first part of
Theorem~\ref{theo:mode_height} does not say \textit{which} of the numbers, $\lfloor u_n^*\rfloor$ or $\lceil u_n^*\rceil$, is the mode. From the proof (which will be given in Section~\ref{subsec:proof_mode_width}) it will be clear that for every given $\eps>0$ we can choose the random variable $N$ such that for every $n>N$ satisfying $\min_{k\in\Z} |u_n-k-\frac 12|>\eps$  the mode is equal to $u_n^*$ rounded to the nearest integer. In the case when $u_n$ is sufficiently close to a half-integer $k+\frac 12$, $k\in\Z$, further terms of the asymptotic expansion are needed to tell whether $u_n$ is $\lfloor u_n^*\rfloor$ or $\lceil u_n^*\rceil$.
It is natural to conjecture that the random variable $W_\infty'(0)/W_\infty(0)$ is absolutely continuous. We would then have the following cases.

\vspace*{1mm}
\noindent
\textit{Case 1.}
If $\varphi'(0)$ is integer, then either $u_n=\lfloor u_n^*\rfloor$ for all $n>N$ or $u_n=\lceil u_n^*\rceil$ for all $n>N$. Indeed, the fractional part of $u_n^*$ is some constant which is not equal to $1/2$ w.p.\ $1$.

\vspace*{1mm}
\noindent
\textit{Case 2.}
If $\varphi'(0)$ is non-integer rational, then for $n>N$, $u_n$ chooses its value among $\lfloor u_n^*\rfloor$ and $\lceil u_n^*\rceil$ periodically; see the top of Figure~\ref{bild:BRW_Argmax} which shows the sequences $\lfloor u_n^*\rfloor$,  $\lceil u_n^*\rceil$ (both black circles) and $u_n$ (red dots).

\vspace*{1mm}
\noindent
\textit{Case 3.} If $\varphi'(0)$ is irrational, then there is no periodicity and both values $\lfloor u_n^*\rfloor$ and $\lceil u_n^*\rceil$ are attained with asymptotic density $\frac 12$; see the bottom of Figure~\ref{bild:BRW_Argmax}.
\end{remark}

\begin{figure}[!htbp]
\includegraphics[width=\textwidth]{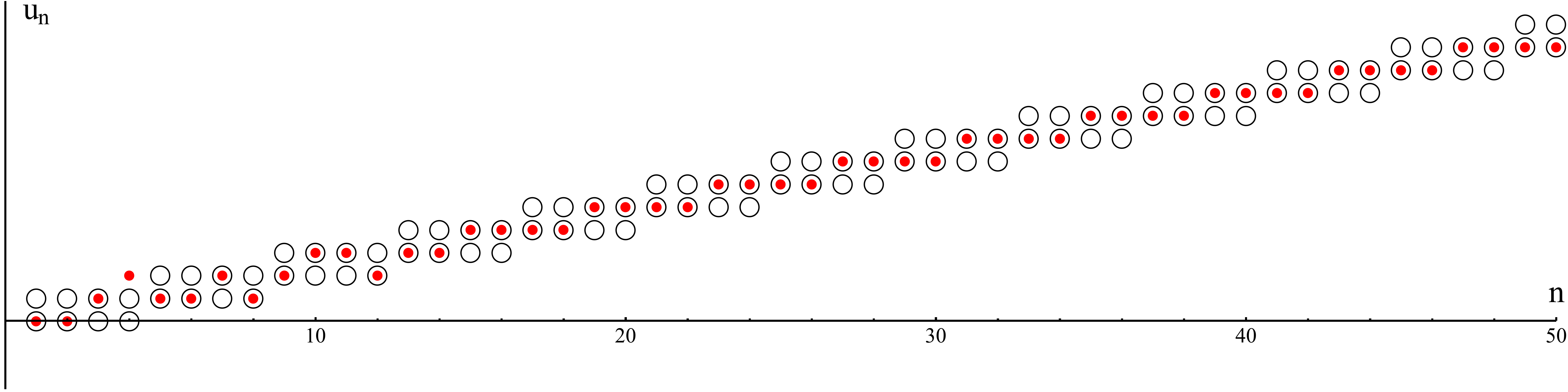}
\includegraphics[width=\textwidth]{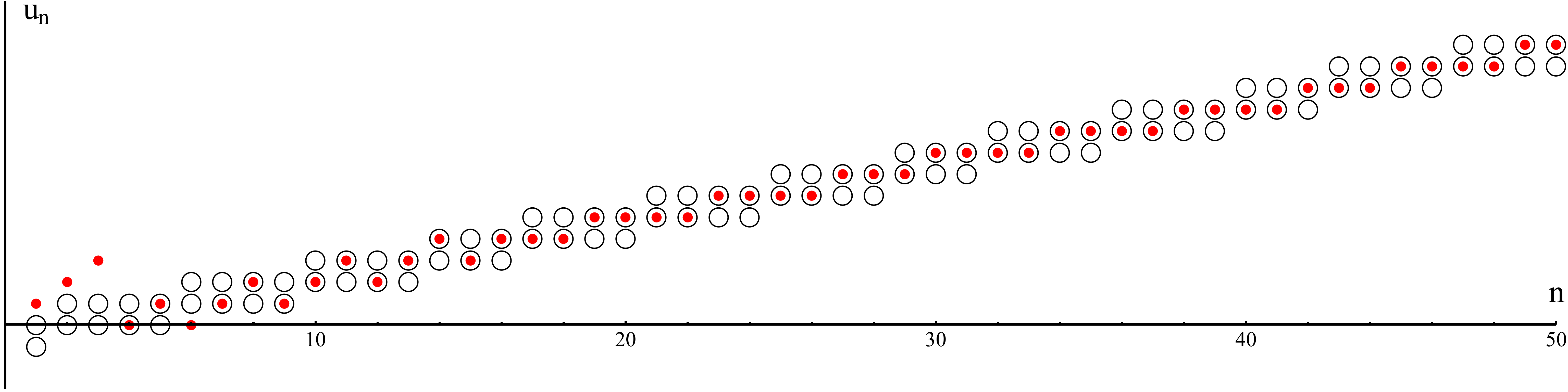}
\caption
{\small A realization of the mode $u_n$ (red dots) as a function of time $n=1,\ldots,50$. Black circles show $\lfloor u_n^*\rfloor$ and $\lceil u_n^* \rceil$. Top: The drift $\varphi'(0)=\frac 14$ is rational, periodic behavior. Bottom: The drift $\varphi'(0)$ is irrational, no periodicity.}
\label{bild:BRW_Argmax}
\end{figure}

\begin{remark}%\label{rem:width_BRW}
Since $|\theta_n| \leq 1/2$,  $\limsup$ and $\liminf$ of the sequence
\begin{equation}\label{eq:width_sequence}
\tilde M_n := 2\sigma^2(0) n\left( 1 - \frac{\sqrt{2\pi n}\, \sigma(0)\, M_n}{W_\infty(0) m^n}\right)
\end{equation}
are a.s.\ finite.
Whether or not the sequence~\eqref{eq:width_sequence} has an a.s.\ limit, depends on the arithmetic properties of $\varphi'(0)$. Let
\begin{equation}\label{eq:Q_def}
Q(\theta) =
\frac{\kappa_3^2(0)}{6\sigma^4(0)} - \frac{\kappa_4(0)}{4\sigma^2(0)} + \theta^2
%+ \frac{W_\infty'^2(0) - W_\infty''(0)W_\infty(0)}{W_\infty^2(0)}
+(\log W_\infty)''(0).
\end{equation}

\vspace*{1mm}
\noindent
\textit{Case 1.}
If $\varphi'(0)$ is integer, then $\theta_n=\theta$ is constant and the a.s.\ limit of~\eqref{eq:width_sequence} exists and equals $Q(\theta)$.

\vspace*{1mm}
\noindent
\textit{Case 2.}
If $\varphi'(0)$ is non-integer but rational, we have finitely many a.s.\ subsequential limits of $\theta_n$ and of~\eqref{eq:width_sequence}. Hence the normalized height~\eqref{eq:width_sequence} has finitely many a.s.\ subsequential limits inside the interval $[Q(0), Q(\frac 12)]$; see the top of Figure~\ref{bild:BRW_Max}. Note that these limits are not equally spaced because of the quadratic term $\theta^2$ in~\eqref{eq:Q_def}.

\vspace*{1mm}
\noindent
\textit{Case 3.}
For irrational $\varphi'(0)$ the set of  subsequential limits of $\theta_n$ is the whole interval $[0,\frac 12]$. The set a.s.\ subsequential limits of~\eqref{eq:width_sequence} is therefore the interval $\{Q(\theta)\colon \theta\in [0,\frac 12]\}$; see the bottom of Figure~\ref{bild:BRW_Max}. It is an interval of length $1/4$. On Figure~\ref{bild:BRW_Max} one sees that there are much more points close to the lower limit $Q(0)$ than to the upper limit $Q(\frac 12)$. To explain this, note that the sequence $\theta_n$ is uniformly distributed on $[0,\frac 12]$ by Weyl's equidistribution theorem (stating that the sequence of fractional parts $\{\varphi'(0) n\}$, $n\in\N$, is uniformly distributed on $[0,1]$), hence the asymptotic density of  $\theta_n^2$ equals $\dd x/\sqrt {2x}$ on $[0,\frac 12]$.
\end{remark}

\begin{figure}[!htbp]
\includegraphics[width=\textwidth]{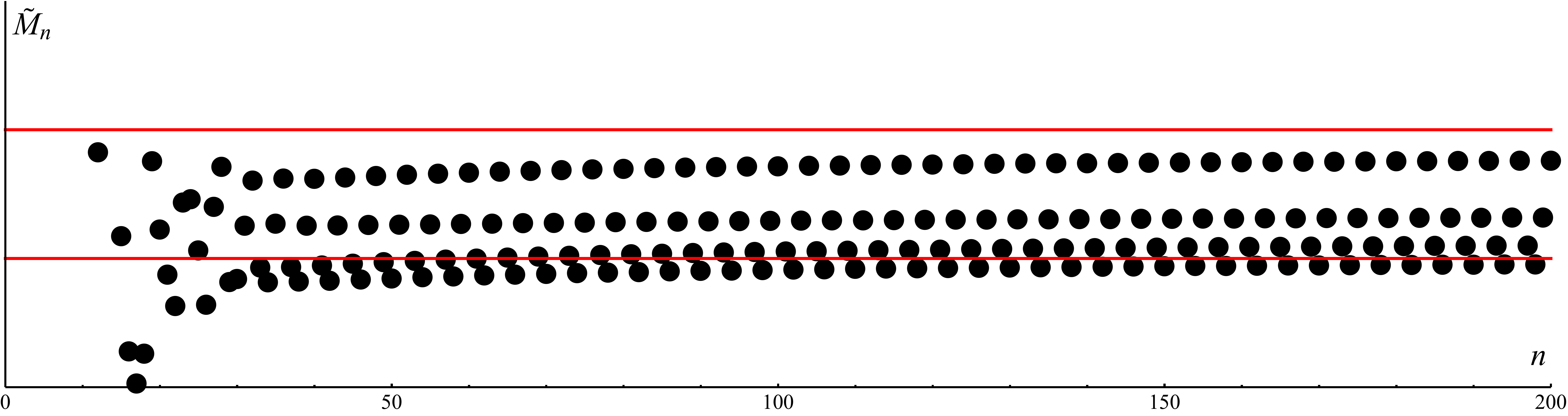}
\includegraphics[width=\textwidth]{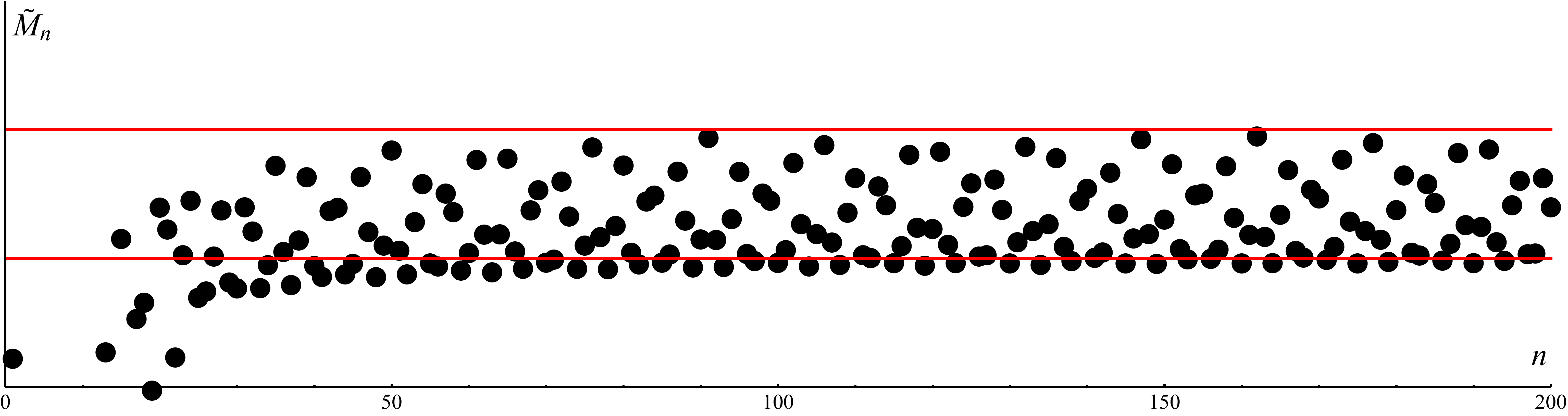}
\caption
{\small The normalized height $\tilde M_n$, see~\eqref{eq:width_sequence}, as a function of time $n=1,\ldots,200$  in the case of rational $\varphi'(0)=\frac 14$ (top) and irrational $\varphi'(0)$ (bottom). The red horizontal lines show the values $Q(0)$ and $Q(\frac 12)$. }
\label{bild:BRW_Max}
\end{figure}

\begin{remark}
For branching random walks in continuous time $t\geq 0$, see Section~\ref{subsec:cont_time_BRW}, we have the following two cases.

\vspace*{1mm}
\noindent
\textit{Case 1.} If $\varphi'(0)=0$, then
$$
\theta := \theta_n = \min_{k\in\Z} \left|\frac{W_\infty'(0)}{W_\infty(0)} - \frac{\kappa_3(0)}{2\sigma^2(0)} -k\right|.
$$
is constant and~\eqref{eq:width_sequence} converges a.s.\ to $Q(\theta)$.

\vspace*{1mm}
\noindent
\textit{Case 2.}
If $\varphi'(0)\neq 0$, then the set of a.s.\ subsequential limits of~\eqref{eq:width_sequence} is  the interval $\{Q(\theta)\colon \theta\in [0,\frac 12]\}$. It is exactly this case which is relevant for random trees.
\end{remark}

\subsection{Related results, extensions and open questions}
Consider the random measure $\pi_n$ on $\Z$ counting the particles in the BRW at time $n$,
$$
\pi_n := \sum_{i=1}^{N_n} \delta_{z_{i,n}}.
$$
Recall that $\delta_z$ denotes the Dirac delta-measure at $z$. The moment generating function of $\pi_n$ is given by
$$
\Lambda_n(\beta) := \int_{\R}\eee^{\beta t} \pi_n(\dd t) = \eee^{\varphi(\beta) n} W_n(\beta),
$$
where $W_n$ is the Biggins martingale defined in~\eqref{eq:biggins_martingale_def}. As $n\to\infty$, we have
\begin{equation}\label{eq:mod_phi}
\sup_{\beta\in\bD_\eps} \left|\eee^{- \varphi(\beta)n} \Lambda_n(\beta) - W_\infty(\beta)\right| \toas 0,
\end{equation}
where $\bD_\eps\subset \C$ is a sufficiently small disc of radius $\eps$ centered at $0$. To use the terminology
introduced recently in~\cite{delbaen_kowalski_nikeghbali}, \cite{feray_meliot_nikeghbali}, the sequence of random
measures $\pi_n$ converges with probability $1$ in the mod-$\varphi$ sense.  Roughly speaking, \eqref{eq:mod_phi} says
that the random measure $\pi_n$ is close to $\vartheta_n$, the $n$-fold convolution of the measure $\vartheta_1$ which has cumulant generating function $\varphi(\beta)$; see~\eqref{eq:def_varphi}. The random analytic function $W_\infty$ describes the ``convolution difference'' between these two measures and therefore plays a central (and somewhat mysterious) role in the theory of mod-$\varphi$ convergence.
Note that in our case, the limiting function $W_\infty$ is \textit{random}. Mod-$\varphi$ convergence is a very strong  notion since it implies many classical limit theorems (like the central and the local limit theorem). In fact, Corollary~\ref{cor:asympt_expansion_BRW_optimal} could be deduced from the general framework recently given in~\citet{feray_meliot_nikeghbali}.

Let us stress that in our proofs we do \textit{not} use the martingale property of $W_n$. We need only that w.p.\ $1$
there is a uniform convergence of $W_n$ to a random analytic function $W_\infty$ and that the speed of this convergence
is superpolynomial; see Lemma~\ref{lem:W_infty_W_T}.
Results similar to those obtained in the present paper should hold for any sequence of ``random profiles'' $L_n:\Z\to\R$ (not necessarily related to the BRW) having a moment generating function that converges (after it has been divided by its expectation) superpolynomially to some random analytic function $W_\infty$.

Our results are strong limit theorems in the sense that they refer to the a.s.\ asymptotic behavior of various random quantities related to $L_n$ such as $L_n(k_n)$, the mode $u_n$ and the height $M_n$.
In the case of $L_n(k_n)$, it is possible to augment our results by more refined asymptotics for the difference between
the quantity
of
interest and its limit. To this end, it suffices to  take more terms in the asymptotic expansion. The a.s.\  limits of the rescaled differences have non-normal distribution that can be expressed through the derivatives of $W_\infty$. %This behavior of $L_n(k_n)$, the mode $u_n$ and the height $M_n$ is in contrast to the behavior of $W_n(\beta)$ (and its derivatives) which was studied in~\cite{roesler_topchii_vatutin1,neininger,gruebel_kabluchko}.

Our main motivation was to apply the results to random trees (which we plan to do in a separate paper), so we did not seek for optimal moment conditions. The moment condition required in Assumption~D can certainly be relaxed.  It should also be possible to extend the results to BRW's on the $d$-dimensional lattice $\Z^d$. On the other hand, it should be possible to obtain an asymptotic expansion for BRW's satisfying the strong non-lattice condition. Indeed, the first two terms were obtained very recently by~\citet{gao_liu}. The expansion stated in Corollary~\ref{cor:asympt_expansion_BRW_optimal}, Equation~\eqref{eq:L_n_uniform_exp}, is valid uniformly in the range of ``central order statistics'', that is as long as $k \in (\varphi'(\beta_-)n+\eps n, \varphi'(\beta_+) n - \eps n)$. The case $k=\varphi'(\beta_\pm)n \mp o(n)$ is not covered by our results because the function $W_\infty (\beta)$ is not analytic at $\beta_+$ and $\beta_-$.

\section{Proofs}\label{sec:proofs}
\subsection{Notation and method of proof}
Consider a branching random walk in discrete time satisfying Assumptions A--E.
%Let $T_1\leq T_2\leq \ldots$ be a sequence of stopping times w.r.t.\ the filtration $\{\cF_n\}_{n\in\N}$ such that
%\begin{equation}\label{eq:lim_T_n_infty}
%\lim_{n\to\infty} T_n=+\infty.
%\end{equation}
Fix some $\beta \in \cD_{\varphi}$ and recall the notation
\begin{equation}
\mu(\beta) =\varphi'(\beta ),
\quad
\sigma^2(\beta) =\varphi''(\beta ).
\end{equation}
Consider a random measure $\nu_n$ counting the appropriately weighted particles of the branching random walk at time $n$:
\begin{equation}\label{eq:def_mu_n}
\nu_{n}=\eee^{-\varphi(\beta)n} \sum_{i=1}^{N_n} \eee^{\beta  z_{i,n}} \delta\left(\frac{z_{i,n} - \mu(\beta) n}{\sigma(\beta)\sqrt{n}}\right).
\end{equation}
Here we have written $\delta(z)$ instead of $\delta_z$ as this is
typographically more convenient.

Note that $\nu_n$ depends on $\beta$ but we usually suppress this in our notation.
%We view $\nu_n$ as a random variable with values in $\cM_+(\R)$, the space of finite Borel measures on $\R$. The space $\cM_+(\R)$ is endowed with the topology of weak convergence.
We have
$$
\nu_n(\R) = W_{n} (\beta ) \toas W_\infty (\beta ),
$$
where $W_n$ is the Biggins martingale given by~\eqref{eq:biggins_martingale_def}.
The measure $\nu_n$ is finite, but, in general, it is not a probability measure.

The characteristic function of the random measure $\nu_n$ is a random function $\{\psi_n(s;\beta )\colon s\in\R\}$ given by
\begin{equation}\label{eq:psi_n_def}
\psi_n(s;\beta ) = \int_\R \eee^{i s z} \nu_n(\dd z)
=
f_{n}(s;\beta) W_{n} \left(\beta +\frac{is}{\sigma(\beta)\sqrt n}\right),
\end{equation}
where
\begin{equation}\label{eq:def_f_n}
f_n (s;\beta ) = \exp\left\{-n\varphi(\beta ) -  \mu(\beta) \sqrt {n}  \frac{is}{\sigma(\beta)}
+ n \varphi\left(\beta +\frac{is}{\sigma(\beta) \sqrt n}\right)\right\}.
\end{equation}

\subsection{Expansion of the characteristic function}\label{sec:proof_exp_char}
In Proposition~\ref{prop:char_funct_asymptotic_expansion} we will state an asymptotic expansion of $\psi_n(s;\beta )$ in powers of $n^{-1/2}$. On the formal level, this expansion is a product of expansions of $f_n$ and $W_n$.  The main idea is that $f_n$ can be represented as a characteristic function of a sum of i.i.d.\ random variables (so that the classical Edgeworth expansion applies), while $W_n$ can be expanded into a Taylor series.  Multiplying the expansions of $f_n$ and $W_n$ we will obtain an expansion of $\psi_n$.

Let us first state an expansion of $f_n$. Define polynomials $\tilde P_j(z;\beta )$, $j\in\N_0$, by the formal identity
$$
\exp \left\{\sum_{j=1}^{\infty} \frac{\kappa_{j+2}(\beta )}{(j+2)!} z^{j+2} u^j\right\} =
\sum_{j=0}^{\infty} \tilde P_j(z;\beta ) u^j.
$$
It is known that $\tilde P_j$ has degree $3j$ and its coefficients can be expressed through $\kappa_3(\beta),\ldots,\kappa_{j+2}(\beta)$; see~\cite[p.~52]{bhattacharya_ranga_rao_book}. Also, $\tilde P_j$ is an even (resp.\ odd) function if $j$ is even (resp.\ odd). The first few polynomials $\tilde P_j$ are given by
$$
\tilde P_0(z;\beta )=1,
\quad
\tilde P_1(z;\beta ) = \frac{\kappa_3(\beta ) z^3}{6},
\quad
\tilde P_2(z;\beta ) = \frac{\kappa_4(\beta ) z^4}{24} + \frac{\kappa_3^2(\beta ) z^6}{72},
$$
$$
\tilde P_3(z;\beta ) = \frac{\kappa_5(\beta ) z^5}{120} + \frac{\kappa_3(\beta )\kappa_4(\beta ) z^7}{144} + \frac{\kappa_3^3(\beta )z^9}{6^4}.
$$
The next lemma is a central element in the proof of the classical Edgeworth expansion; see Lemma~11 in~\cite[Ch.~VI, p.~175]{petrov_book}.
\begin{lemma}\label{lem:est_f_T}
Fix $r\in\N_0$ and a compact interval $K\subset \cD_\varphi$. For all $n\in\N$, real $|s| < n^{1/7}$, and $\beta\in K$, we have
$$
\left|f_n(s;\beta ) - \eee^{-\frac 12 {s^2}} \sum_{j=0}^r \frac{1}{n^{j/2}} \tilde P_j \left(\frac{is}{\sigma(\beta)}; \beta \right)\right| \leq \frac{\delta_1(n)}{n^{r/2}} \eee^{-\frac 12 {s^2}} (1 + |s|^{3r+3}),
$$
where $\delta_1(n)$ does not depend on $s$ and $\lim_{n\to+\infty}\delta_1(n)=0$.
\end{lemma}
\begin{proof}
Recall that $\vartheta_1$ is the intensity of the branching random walk at time $n=1$.  Let $Z_1,Z_2,\ldots$ be i.i.d.\ random integer-valued random variables with probability distribution
\begin{equation}\label{eq:Z_1_distr}
\P[Z_1=k] := \eee^{-\varphi(\beta )+ \beta  k} \vartheta_1(\{k\}), \quad k\in\Z.
\end{equation}
Then the log-characteristic function of $Z_1$ is just
\begin{equation}\label{eq:Z_1_char_func}
\log \E [\eee^{is Z_1}] = \varphi(\beta +is) - \varphi(\beta ) = \sum_{j=1}^{\infty} \kappa_j(\beta ) \frac{(is)^j}{j!},
\end{equation}
where the second equality holds for sufficiently small $|s|$.
In particular, $\E Z_1 = \mu(\beta) $, $\Var Z_1 =\sigma^2(\beta)$. Consider the random variable
$$
S_n^* := \frac{Z_1+\ldots+Z_n - \mu(\beta)  n}{\sigma(\beta) \sqrt n}.
$$
The characteristic function of $S_n^*$ is $f_n(s;\beta )$; see~\eqref{eq:def_f_n} and~\eqref{eq:Z_1_char_func}.  The moments of $Z_1$ are finite because $\beta\in \cD_\varphi$ and $\cD_\varphi$ is open. Hence we can apply Lemma~11 from~\cite[Ch.~VI, p.~175]{petrov_book}. In fact, by looking at the proof of this lemma, one can see that $\lim_{n\to\infty} \delta_1(n) = 0$ holds uniformly in $\beta\in K$.   This yields the statement of the lemma.
\end{proof}
At this place, we can prove Proposition~\ref{prop:expansion_expected}.
\begin{proof}[Proof of Proposition~\ref{prop:expansion_expected}]
Let $Z_1,Z_2,\ldots$ be i.i.d.\ integer-valued random variables with probability distribution~\eqref{eq:Z_1_distr} and log-characteristic function~\eqref{eq:Z_1_char_func}.
%$$
%\P[Z_1=k] := \eee^{-\varphi(\beta )+ \beta  k} \vartheta_1(\{k\}), \quad k\in\Z.
%$$
%Then the characteristic function of $Z_1$ is just
%$$
%\E \eee^{is Z_1} = \varphi(\beta +is) - \varphi(\beta ) = \sum_{k=1}^{\infty} \kappa_k(\beta ) \frac{(is)^k}{k!}.
%$$
%In particular, $\E Z_1 = \mu $, $\Var Z_1 =\sigma ^2$, and t
The $j$-th cumulant of $Z_1$ is $\kappa_j(\beta )$. We have
$$
\eee^{\beta  k - \varphi(\beta )n} \E L_n(k) = \eee^{\beta  k - \varphi(\beta )n} \vartheta_n(\{k\}) = \P [Z_1+\ldots+Z_n = k]
$$
Applying the classical asymptotic expansion~\eqref{eq:asympt_expansion_random_walk}, we obtain~\eqref{eq:asympt_expansion_intensity}. By examining the proof of the classical expansion, Theorem~13 in~\cite[Ch.~VII, p.~205]{petrov_book}, one can see that all estimates there are uniform in $\beta\in K$.
\end{proof}

The next lemma is a Taylor expansion of $W_\infty$. Recall that $(\beta_-,\beta_+)$  denotes the open interval consisting of all $\beta \in \cD_\varphi$  such that $\varphi'(\beta )\beta  < \varphi(\beta )$.

\begin{lemma}\label{lem:W_infty_Taylor}
Fix $r\in\N_0$ and a compact interval $K \subset (\beta_-,\beta_+)$. Then we can find a sufficiently small $\eps>0$ such that for all $\beta\in K$ and every $u\in \C$ with $|u|\leq \eps$,
$$
\left|W_{\infty}(\beta  + u) - \sum_{j=0}^r \frac{W_\infty^{(j)}(\beta)}{j!} u^j \right| \leq  M_1 |u|^{r+1},
$$
where $M_1$ is an a.s.\ finite random variable.
\end{lemma}

\begin{proof}
By~\cite[Corollary 3]{biggins_uniform}, for sufficiently small $\eps>0$ the random function $W_{\infty} (\beta)$ is analytic on the domain $K_{2\eps} := \{\beta\in\C\colon \dist (\beta, K) <2\eps\}$, with probability $1$. Fix some $\beta\in K$. The random function
$$
\eta(u;\beta) := \frac 1 {u^{r+1}} \left(W_{\infty}(\beta  + u) - \sum_{j=0}^r \frac{W_\infty^{(j)}(\beta)}{j!} u^j\right),
\;\;\;
0<|u|\leq \eps,
$$
can be extended to $u=0$ by continuity. It is a.s.\ analytic, hence the maximum principle yields that
$$
|\eta(u;\beta)| \leq \sup_{|u|=\eps} |\eta(u;\beta)|
\leq
\frac 1{\eps^{r+1}}\left(\sup_{|u|=\eps} |W_\infty(\beta+u)| +  \sum_{j=0}^r \frac{|W_\infty^{(j)}(\beta)|}{j!} \eps^j \right).
$$
The right-hand side can be bounded by an a.s.\ finite random variable $M_1$ not depending on $\beta\in K$, thus proving the statement.
%Subtracting from $W_{\infty}(\beta_0 +u)$ the first$
\end{proof}

The next lemma shows that $W_n$ converges to $W_\infty$ at an exponential rate.
The proof closely follows the arguments given in Section 3 of~\cite{biggins_uniform}; see
also~\cite[Section 3]{biggins_uniform_IMS}. It is here that Assumption~D is crucial as it allows
us to move into the complex plane and to use Cauchy's integral formula in order to obtain a suitable
maximal inequality, see~\eqref{eq:exprate2} below. Such an inequality is the basis for
sup-norm convergence in the space of continuous functions on a compact set.

\begin{lemma}\label{lem:W_infty_W_T}
Fix a compact interval $K \subset (\beta_-,\beta_+)$.  Then we can find a sufficiently small
$\eps>0$ such that for all $n\in\N$,
$$
%\sup_{\beta\in \C\colon |\beta-\beta | < \eps}
\sup_{\beta\in K} \sup_{|u|\leq \eps}
|W_n(\beta +u) - W_\infty(\beta +u)| < \eee^{-\eps n} M_2,
$$
where $M_2$ is some a.s.\ finite random variable.
\end{lemma}
\begin{proof}
Let $\beta\in K$. The first part of Lemma~2 in~\cite{biggins_uniform} implies that for some $\rho>0$, and with
\begin{equation*}
  z_1(t) \, :=\,  \beta + 2\rho \eee^{2\pi i t},\ 0\le t \leq 1,
\end{equation*}
we have that, for some $C=C(\beta,\rho)<\infty$ and some $\kappa=\kappa(\beta,\rho)<1$,
\begin{equation}\label{eq:exprate1}
  \E\left| W_{n+1}\bigl(z_1(t)\bigr) - W_{n}\bigl(z_1(t)\bigr)\right| \, \le C\, \kappa^n
    \quad\text{for all } n\in\bN,\, 0\le t\le 1.
\end{equation}
Further, let %$z_2(t) :=  \beta + 2\rho \eee^{2\pi t}$ and
$\bD_\rho(\beta)=\{z\in\bC:\,|z-\beta|\le \rho\}$.
Then, by Lemma~3 in~\cite{biggins_uniform},
\begin{equation}\label{eq:exprate2}
   \sup_{z\in \bD_\rho(\beta)} \bigl|W_{n+1}(z)-W_n(z)\bigr|
      \; \le\; \int_0^1  \Bigl| W_{n+1}\bigl(z_1(t)\bigr) - W_{n}\bigl(z_1(t)\bigr)\Bigr| \, \dd t.
\end{equation}
Combining~\eqref{eq:exprate1} and~\eqref{eq:exprate2} and using Fubini's theorem we arrive at
\begin{align*}
  \E\left[\sup_{z\in \bD_\rho(\beta)} \bigl|W_{n+1}(z)-W_n(z)\bigr|\right]\ \le \
      \int_0^1  \E\Bigl| W_{n+1}\bigl(z_1(t)\bigr) - W_{n}\bigl(z_1(t)\bigr)\Bigr| \, \dd t
         \ \le \ C\, \kappa^n.
\end{align*}
As $K$ is compact we can now find an $\eps>0$ such that
\begin{equation*}
  \sum_{n=1}^\infty \eee^{\eps n}\, \E\biggl[\sup_{\beta\in K}
          \sup_{|u|\le \eps} \bigl|W_{n+1}(\beta+u)-W_n(\beta+u)\bigr|\biggr]
    \; < \; \infty,
\end{equation*}
and we may then take
\begin{equation*}
  M_2\; :=\;
       \sum_{n=1}^\infty \eee^{\eps n} \sup_{\beta\in K}
           \sup_{|u|\le\eps} \bigl|W_{n+1}(\beta+u)-W_n(\beta+u)\bigr|. \qedhere
\end{equation*}
\end{proof}

Using these tools we obtain
\begin{proposition}\label{prop:est_W_T}
Fix $r\in\N_0$ and a compact interval $K \subset (\beta_-,\beta_+)$.  Then we can find a sufficiently small  $\eps>0$ such that for every $\beta\in K$,  $n\in\N$ and $s\in\C$ satisfying $|s|\leq \eps \sigma(\beta)  \sqrt n$ we have
$$
\left|W_{n}\left(\beta  + \frac{is}{\sigma(\beta) \sqrt n}\right)
-
\sum_{j=0}^r \frac{1}{n^{j/2}}\frac{W_\infty^{(j)}(\beta)}{j!} \left(\frac{is}{\sigma(\beta) }\right)^j\right|
<
\frac{M_3}{n^{\frac {r+1}2}} (1+ |s|^{r+1}),
%M_1 |u|^{r+1} +  \eee^{-\eps n} M_2.
$$
where $M_3$ is some a.s.\ finite random variable.
\end{proposition}
\begin{proof}
Apply Lemmas~\ref{lem:W_infty_Taylor} and~\ref{lem:W_infty_W_T} with $u=is/(\sigma(\beta) \sqrt n)$.
\end{proof}

The next proposition states the asymptotic expansion of the characteristic function $\psi_n(s;\beta)$.
\begin{proposition}\label{prop:char_funct_asymptotic_expansion}
Fix $r\in\N_0$ and a compact interval $K\subset (\beta_-,\beta_+)$. Then, for sufficiently large $n\in\N$, $\beta \in K$, and all real $|s|<n^{1/7}$ we have
\begin{equation}\label{eq:delta_r_est}
\Delta_r(s,n;\beta ) := \left|\psi_n(s;\beta ) - \eee^{-\frac 12 {s^2}} V_{r,n}(s;\beta )\right|
<
\frac{\delta_2(n)}{n^{r/2}} \eee^{-\frac 12 s^2} (1 + |s|^{4r+3}),
\end{equation}
where
\begin{equation}\label{eq:U_tilde_def}
V_{r,n}(s;\beta )
=
\sum_{j=0}^r \frac{1}{n^{j/2}} \sum_{m=0}^{j} \frac{W_{\infty}^{(m)}(\beta )}{m!}\left(\frac {is}{\sigma(\beta) }\right)^m\tilde P_{j-m}\left(\frac{is}{\sigma(\beta) };\beta \right)
\end{equation}
and $\delta_2(n)$ is a sequence of random variables (not depending on $s$ or $\beta $) such that $\lim_{n\to+\infty}\delta_2(n)=0$ a.s.
\end{proposition}
\begin{proof}
%Take some $\beta_0\in K$.
Let $\eps>0$ be sufficiently small to ensure that the statement of Proposition~\ref{prop:est_W_T} holds.
%It suffices to prove that the estimate~\eqref{eq:delta_r_est} holds uniformly in $\beta\in \bar \bD_\eps(\beta_0)$ because by compactness of $K$ we can cover $K$ by finitely many such disks.
The subsequent estimates are uniform in $\beta\in K$.
We use the inequality
$$
|A_1B_1-A_2B_2| \leq |A_1-A_2| |B_1| + |A_2||B_1-B_2|,
$$
which is an easy consequence of the triangle inequality, with
\begin{alignat*}{2}
A_1 &= f_n(s;\beta ),
&\quad
B_1 &= W_{n}\left(\beta  + \frac{is}{\sigma(\beta) \sqrt n}\right),\\
A_2 &= \eee^{-\frac 12 {s^2}} \sum_{j=0}^r \frac{1}{n^{j/2}} \tilde P_j \left(\frac{is}{\sigma(\beta) }; \beta \right),
&\quad
B_2 &= \sum_{j=0}^r \frac{1}{n^{j/2}}\frac{W_\infty^{(j)}(\beta )}{j!} \left(\frac{is}{\sigma(\beta) }\right)^j.
\end{alignat*}
Recall from~\eqref{eq:psi_n_def} that $\psi_n(s;\beta )= A_1B_1$. Condition $|s|<n^{1/7}$ implies that (for sufficiently
large $n$) we have $|s|\leq \eps \sigma(\beta)  \sqrt n$ and hence Lemma~\ref{lem:est_f_T} and Proposition~\ref{prop:est_W_T} yield the estimates
$$
|A_1-A_2| \leq  \frac{\delta_1(n)}{n^{r/2}} \eee^{-\frac 1 2 {s^2}} (1 + |s|^{3r+3}),
\quad
|B_1-B_2| \leq \frac{M_3}{n^{\frac{r+1} 2}} (1+ |s|^{r+1}).
$$
In particular, since $\tilde P_j$ has degree $3j$ and $|s|<n^{1/7}$, we have the estimates
$$
%\left|W_{n}\left(\beta  + \frac{is}{\sigma \sqrt n}\right)\right| < M_4,
|B_2|< M_4,
\quad
|B_1| < M_4,
\quad
|A_2| < C \eee^{-\frac 12 {s^2}},
$$
where $M_4$ is a.s.\ finite random variable and $C>0$ is a constant. By the above estimates it follows that
\begin{align}
\left| \psi_n(s;\beta ) -A_2B_2 \right|
&\leq
\frac{\delta_1(n)}{n^{r/2}} \eee^{-\frac 12 {s^2}} (1 + |s|^{3r+3}) M_4 + C \eee^{-\frac 12 s^2}
\frac{M_3}{n^{\frac {r+1}2}} (1+ |s|^{r+1})\label{eq:tech_psi_exp_1}\\
&\leq
\frac{\delta_3(n)}{n^{r/2}} \eee^{-\frac 12 {s^2}} (1 + |s|^{3r+3}),\notag
\end{align}
where $\delta_3(n)$ is an a.s.\ finite random variable and $\lim_{n\to+\infty}\delta_3(n)=0$ a.s. Also, in the double sum
$$
A_2B_2 = \eee^{-\frac 12 s^2}
\sum_{l=0}^r \sum_{m=0}^r
\frac{1}{n^{(l+m)/2}} \tilde P_l \left(\frac{is}{\sigma(\beta) }; \beta \right) \frac{W_\infty^{(m)}(\beta )}{m!} \left(\frac{is}{\sigma(\beta) }\right)^m
$$
the sum of all terms having $j := l + m \leq r$ is $\eee^{-\frac 12 s^2} V_{r,n} (s;\beta )$. The sum of the remaining terms can be estimated by
\begin{equation}\label{eq:tech_psi_exp_2}
|A_2B_2 - \eee^{-\frac 12 s^2} V_{r,n} (s;\beta )| < \frac{M_5}{n^{\frac {r+1}2}} \eee^{-\frac 12 s^2} (1+|s|^{4r}) .
\end{equation}
The statement of the proposition follows from~\eqref{eq:tech_psi_exp_1} and~\eqref{eq:tech_psi_exp_2}.
\end{proof}

\subsection{Proof of Theorem~\ref{theo:asympt_expansion_BRW}}\label{sec:prooftheo2.1}
Recall from~\eqref{eq:x_n_k_def} that
$$
x_n(k)=\frac{k-\mu(\beta) n}{\sigma(\beta) \sqrt{n}}.
$$
We can write the characteristic function of $\nu_n$, see~\eqref{eq:psi_n_def}, as follows:
$$
\psi_n(s;\beta ) = \eee^{-\varphi(\beta )n} \sum_{k\in\Z}  \eee^{is x_n(k)} \eee^{\beta  k} L_n(k).
$$
This function is periodic in $s$ with period $2\pi \sigma(\beta)\sqrt n$. Inverting the Fourier transform we obtain that for $k\in\Z$,
\begin{equation}\label{eq:L_T_k_Fourier}
\sigma(\beta)  \sqrt{n}\, \eee^{\beta k-\varphi(\beta)n} L_n(k)
=
\frac 1{2\pi} \int_{-\pi \sigma(\beta)  \sqrt n}^{+\pi \sigma(\beta)  \sqrt n} \psi_n(s;\beta ) \eee^{- is x_n(k)} \dd s.
\end{equation}
Finally, we are ready to prove the asymptotic expansion of $L_n(k)$. To this end, we will insert the expansion of $\psi_n$ into~\eqref{eq:L_T_k_Fourier}. Let $D=\frac{\dd}{\dd z}$ be the differentiation operator. Define polynomials $Q_{m,j}$, $j\in\N_0$, $m=0,\ldots,j$, by
\begin{equation}\label{eq:Q_m_j_def}
\left(-\frac{D}{\sigma(\beta) }\right)^m \tilde P_{j-m} \left(-\frac{D}{\sigma(\beta) };\beta \right)  \eee^{-\frac 12 z^2}
=
Q_{m,j} (z;\beta)  \eee^{-\frac 12 z^2}.
\end{equation}
Note that the function $q_j(z;\beta)$ appearing in the classical Edgeworth expansion~\eqref{eq:asympt_expansion_intensity} is given by
\begin{equation}\label{eq:q_j_formula}
\tilde P_{j} \left(-\frac{D}{\sigma(\beta) };\beta \right)  \eee^{-\frac 12 z^2}
=
q_{j} (z;\beta)  \eee^{-\frac 12 z^2},
\end{equation}
so that $q_j(z;\beta) = Q_{0,j}(z;\beta)$.
\begin{remark}\label{rem:odd}
It follows that both $Q_{m,j}(z;\beta)$ and $q_j(z;\beta)$ are linear combinations of Hermite polynomials $\Herm_k(z)$ with coefficients depending on $\beta$. For even (resp.\ odd) $j$, all Hermite polynomials involved have even (resp.\ odd) index $k$ and hence the functions $Q_{m,j}$ and $q_j$ are even (resp.\ odd).
%Indeed, for odd $j$ any term in the formula for both $F_j(0;\beta)$ and $q_j(0;\beta)$, see~\eqref{eq:Q_m_j_def} and~\eqref{eq:q_j_formula} below, contains $\Herm_k(0)$ with odd $k$ and hence vanishes.
\end{remark}

Fix $r\in\N_0$ and a compact interval $K \subset (\beta_-,\beta_+)$. Our aim is to prove that
\begin{equation}\label{eq:proof_asympt_main}
     \sup_{\beta\in K} \sup_{k\in\Z}n^{\frac {r}2}
               \left|\sigma(\beta)\sqrt n\, \eee^{\beta k-\varphi(\beta)n} L_n(k)
                       -
      \frac{\eee^{-\frac 12 x_n^2(k)}}{\sqrt{2\pi}} U_{r,n}(x_n(k);\beta )
     \right| \toas 0,
\end{equation}
where
\begin{equation}\label{eq:U_r_n_def}
U_{r,n}(z;\beta )
=
\sum_{j=0}^{r} \frac{1}{n^{j/2}} \sum_{m=0}^j \frac{W_\infty^{(m)}(\beta)}{m!} Q_{m,j}(z;\beta ).
\end{equation}
The idea is that on the right-hand side of~\eqref{eq:L_T_k_Fourier} we can approximate $\psi_n(s;\beta)$ by $\eee^{-\frac 12 s^2} V_{r,n}(s;\beta)$  and the integration range can be replaced by $\R$. Namely, we will show below that
\begin{equation}\label{eq:tech_main}
n^{\frac r2} \sup_{\beta\in K} \sup_{k\in\Z} \left|\int_{-\pi \sigma(\beta)  \sqrt n}^{+\pi \sigma(\beta)  \sqrt n} \psi_n(s;\beta ) \eee^{- is x_n(k)} \dd s - \int_{\R} \eee^{-\frac 12 s^2} V_{r,n}(s;\beta)\eee^{- is x_n(k)} \dd s\right| \toas 0.
\end{equation}
The main tool in the proof of~\eqref{eq:tech_main} is Proposition~\ref{prop:char_funct_asymptotic_expansion}. Given that~\eqref{eq:tech_main} holds, we can replace the integral on the right-hand side of~\eqref{eq:L_T_k_Fourier} by
$$
\frac 1 {2\pi} \int_{\R} \eee^{-\frac 12 s^2} V_{r,n}(s;\beta ) \eee^{-is x_n(k)} \dd s = \frac 1{\sqrt{2\pi}} \eee^{-\frac 12 x_n^2(k)} U_{r,n}(x_n(k);\beta ),
$$
where $U_{r,n}(z;\beta)$ is as in~\eqref{eq:U_r_n_def}. To see this, recall~\eqref{eq:U_tilde_def} and note that  for all $l\in\N_0$,
\begin{align*}
\frac{1}{2\pi} \int_{\R} \eee^{-\frac 12 s^2} \left(\frac{is}{\sigma(\beta) }\right)^l \eee^{-isz} \dd s
&=
\frac{1}{2\pi} \int_{\R} \eee^{-\frac 12 s^2} \left(-\frac{D}{\sigma(\beta) }\right)^l \eee^{-isz} \dd s\\
&=
\frac 1{\sqrt {2\pi}} \left(-\frac{D}{\sigma(\beta) }\right)^l  \eee^{-\frac 12 z^2}.
\end{align*}
Further,
\begin{equation*}
  s\ \mapsto \ \left(\frac{is}{\sigma(\beta)}\right)^m \tilde P_{j-m} \left(\frac{is}{\sigma(\beta)};\beta\right) \eee^{-\frac 12 s^2}
\end{equation*}
is the Fourier transform of the function
\begin{equation*}
  z\ \mapsto\  \frac 1 {\sqrt{2\pi}} \left(-\frac{D}{\sigma(\beta) }\right)^m \tilde P_{j-m} \left(-\frac{D}{\sigma(\beta) };\beta \right)  \eee^{-\frac 12 z^2}.
\end{equation*}
Hence, using ~\eqref{eq:Q_m_j_def} and Fourier inversion,
$$
\frac{1}{2\pi} \int_{\R} \eee^{-\frac 12 s^2} \left(\frac{is}{\sigma(\beta)}\right)^m \tilde P_{j-m} \left(\frac{is}{\sigma(\beta)};\beta\right) \eee^{-isz} \dd s
=
\frac 1{\sqrt {2\pi}} Q_{m,j}(z;\beta) \eee^{-\frac 12 z^2}.
$$
In the rest of the proof we establish~\eqref{eq:tech_main}.

\vspace*{2mm}
\noindent
\textsc{Step 1.} It follows from Proposition~\ref{prop:char_funct_asymptotic_expansion} that
\begin{equation}\label{eq:proof_main_tech1}
n^{r/2} \sup_{\beta\in K} \int_{-n^{1/7}}^{n^{1/7}} \left|\psi_n(s;\beta ) - \eee^{-\frac 12 {s^2}} V_{r,n}(s;\beta )\right| \dd s  \toas 0.
\end{equation}

\vspace*{2mm}
\noindent
\textsc{Step 2.} We prove that for any $a>0$ small enough
\begin{equation}\label{eq:proof_main_tech1a}
n^{r/2} \sup_{\beta\in K} \int_{n^{1/7} \leq  |s|\leq a \sqrt n} |\psi_n(s;\beta)| \dd s  \toas 0.
\end{equation}
%Let $T=n$ be integer.
The functions $W_n$, $n\in\N$, can be bounded by an a.s.\ finite random variable in some neighborhood of $K$. By~\eqref{eq:psi_n_def}, it is therefore sufficient to prove that
\begin{equation}\label{eq:proof_main_tech1a1}
n^{r/2} \sup_{\beta\in K} \int_{n^{1/7} \leq  |s|\leq a \sqrt n} |f_n(s;\beta)| \dd s  \to 0.
\end{equation}
Recall that $f_n(s;\beta)$ is the characteristic function of $(Z_1^*+\ldots+Z_n^*)/\sqrt n$, where $Z_i^*:=(Z_i-\mu(\beta) )/\sigma(\beta) $ are i.i.d.\ random variables with distribution given in~\eqref{eq:Z_1_distr}. Note that $\E Z_i^*=0$, $\Var Z_i^*=1$, and $c:=\E |Z_i^*|^3$ is bounded as long as $\beta\in K$. By Lemma~12 from~\cite{petrov_book} (where we take $b=1/2$) we have, for $|s|<\sqrt n/(2c)$, the estimate
$$
\sup_{\beta\in K} |f_n(s;\beta )| \leq \eee^{-\frac 16 s^2}.
$$
This implies~\eqref{eq:proof_main_tech1a1} with any $a<1/(2c)$.

\vspace*{2mm}
\noindent
\textsc{Step 3.} We prove that for a sufficiently small $a>0$,
\begin{equation}\label{eq:proof_main_tech1b}
n^{r/2} \sup_{\beta\in K} \int_{a\sqrt n \leq  |s|\leq \pi \sigma(\beta)  \sqrt n} |\psi_n(s;\beta)| \dd s  \toas 0.
\end{equation}
The main difficulty of this step is that, on the range in~\eqref{eq:proof_main_tech1b}, we cannot claim that the functions $W_n$, $n\in\N$, are  uniformly bounded by some a.s.\ finite random variable. Instead, we shall employ an estimate on the moments of $W_n$ obtained in Lemma~2 of~\cite{biggins_uniform}.
We consider only $s\in [a\sqrt n, \pi \sigma(\beta)\sqrt n]$, the case of negative $s$ being analogous.  Write $u:= s/(\sigma(\beta) \sqrt n) \in [a/\sigma(\beta), \pi]$ and $z:= \beta + iu$, so that $\beta =\Re z$. By~\eqref{eq:psi_n_def} and~\eqref{eq:def_f_n} we have
$$
|\psi_n(s;\beta)| = \frac{|W_n(z)|} {\eee^{n \varphi(\Re z) - n \Re \varphi(z)}} .
$$
Let $a_0 = \inf_{\beta \in K} (a/\sigma(\beta))>0$. To establish~\eqref{eq:proof_main_tech1b} it suffices to prove that
\begin{equation}\label{eq:proof_step4_tech1}
n^{(r+1)/2} \sup_{\beta\in K} \sup_{a_0 \leq  u \leq \pi}  \frac{|W_n(z)|}{\eee^{n \varphi(\Re z) - n \Re \varphi(z)}}  \toas 0.
\end{equation}

We shall show that for every $z_0\in\C$ such that $\Re z_0\in (\beta_-,\beta_+)$ and $\Im z_0 \in [a_0, \pi]$, there are  sufficiently small $\eps_0,\delta_0>0$ and sufficiently large $C_0$ (all quantities depending on $z_0$) such that for all $n\in\N$,
\begin{equation}\label{eq:proof_step4_tech2}
%M :=
\E \left[ \sup_{z\in\bD_{\eps_0}(z_0)}  \frac{|W_n(z)|}{\eee^{n \varphi(\Re z) - n \Re \varphi(z)}}\right]  < C_0 \eee^{-\delta_0 n},
\end{equation}
where $\bD_{\eps_0}(z_0) = \{z\in \C\colon |z-z_0| < \eps_0\}$.
Indeed, given~\eqref{eq:proof_step4_tech2}, we can cover the compact set $K\times [a_0,\pi]\subset \C$ by finitely many disks of the form $\bD_{\eps_0}(z_0) $, and then use Markov's inequality to obtain~\eqref{eq:proof_step4_tech1}.

In the following we prove~\eqref{eq:proof_step4_tech2}. We have $\varphi'(\Re z_0) \Re z_0 <\varphi(\Re z_0)$  by the assumption  $\Re z_0\in (\beta_-,\beta_+)$ and the definition of the interval $(\beta_-,\beta_+)$ given in~\eqref{eq:beta_-}, \eqref{eq:beta_+}. Hence,  we can find $\alpha\in (1,2)$ close to $1$ such that $\alpha \Re z_0 \in (\beta_-,\beta_+)$ and  $\alpha \varphi(\Re z_0) - \varphi(\alpha \Re z_0) >0$. Further, by Assumption~E, $\varphi(\Re z_0)>\Re \varphi(z_0)$ because $\Im z_0 \in [a_0,\pi]$ and $a_0>0$.  (This is the only place in the proof where we use Assumption~E).  Therefore, we can choose $\delta_0>0$ such that
\begin{equation}\label{eq:delta_0_def}
2\delta_0 < \min\{\varphi(\Re z_0) - \alpha^{-1}\varphi(\alpha \Re z_0), \varphi(\Re z_0) - \Re \varphi(z_0)\}.
\end{equation}
For $z\in\C$ such that $\alpha \Re z \in (\beta_-,\beta_+)$ define
$$
\kappa(z) = \eee^{\varphi(\alpha \Re z) - \alpha \Re \varphi(z)}.
$$
By the continuity of the functions $\varphi$ and $\kappa$, for $\eps_0>0$ small enough, the disk $\bD_{3\eps_0} (z_0)$ is contained in the strip $\{z\in\C\colon \alpha\Re z \in (\beta_-,\beta_+)\}$ and for all $z\in \bD_{3\eps_0} (z_0)$,
\begin{equation}\label{eq:proof_step4_est1}
\eee^{n \varphi(\Re z) - n \Re \varphi(z)}
\geq
\eee^{n \varphi(\Re z_0) - n \Re \varphi(z_0) - n\delta_0/3},
\quad
\kappa(z) \leq \eee^{\delta_0/3} \kappa(z_0).
\end{equation}
By the Cauchy integral formula, see Lemma~3 in~\cite{biggins_uniform},
\begin{equation}\label{eq:proof_step4_est2}
\sup_{z\in\bD_{\eps_0}(z_0)}  |W_n(z)|
%\leq
%\eee^{- n \varphi(\Re z_0) + n \Re \varphi(z_0) + n\delta_0/2} \sup_{z\in\bD_{\eps_0}(z_0)}  |W_n(z)|
\leq
\frac 1 \pi \int_{0}^{2\pi}  |W_n(z_0 + 2\eps_0 \eee^{i \theta})| \dd \theta.
\end{equation}
Combining estimates~\eqref{eq:proof_step4_est1} and~\eqref{eq:proof_step4_est2} and using Jensen's inequality afterwards, we obtain
\begin{align*}
%\E \left[ \sup_{z\in\bD_{\eps_0}(z_0)} \eee^{- n \varphi(\Re z) + n \Re \varphi(z)} |W_n(z)|\right]
\E \left[ \sup_{z\in\bD_{\eps_0}(z_0)}  \frac{|W_n(z)|}{\eee^{n \varphi(\Re z) - n \Re \varphi(z)}}\right]
&\leq
\frac {\int_{0}^{2\pi} \E  |W_n(z_0 + 2\eps_0 \eee^{i \theta})| \dd \theta} {\pi\eee^{n \varphi(\Re z_0) - n \Re \varphi(z_0) - n\delta_0/3}}\\
&\leq
\frac {\int_{0}^{2\pi} (\E  |W_n(z_0 + 2\eps_0 \eee^{i \theta})|^{\alpha})^{1/\alpha} \dd \theta} {\pi  \eee^{n \varphi(\Re z_0) - n \Re \varphi(z_0) - n\delta_0/3}}.
\end{align*}
By Lemma 2(ii) from~\cite{biggins_uniform} (which uses Assumption~D), uniformly over $\theta\in [0,2\pi]$ it holds that
\begin{multline*}
(\E  |W_n(z_0 + 2\eps_0 \eee^{i \theta})|^{\alpha})^{1/\alpha}
\leq C \left(\sum_{r=0}^{n-1} \kappa^r(z_0 + 2\eps_0 \eee^{i \theta})\right)^{1/\alpha}\\
\leq C  \eee^{n\delta_0/3} \left(\sum_{r=0}^{n-1} \kappa^r(z_0)\right)^{1/\alpha}
\leq C' \eee^{n\delta_0/2} (\kappa(z_0) \vee 1)^{n/\alpha},
\end{multline*}
where in the second inequality we applied~\eqref{eq:proof_step4_est1}.
From the above two estimates and the definition of $\kappa$ it follows that
\begin{multline*}
\E \left[ \sup_{z\in\bD_{\eps_0}(z_0)}  \frac{|W_n(z)|}{\eee^{n \varphi(\Re z) - n \Re \varphi(z)}}\right]\\
\leq
C'' \eee^{n\delta_0}  \left( \eee^{\alpha^{-1}\varphi(\alpha \Re  z_0) - \varphi(\Re z_0)} \vee \eee^{\Re \varphi(z_0) - \varphi(\Re z_0)}\right)^{n}
\leq
C''\eee^{-n\delta_0},
\end{multline*}
where in the last step we used~\eqref{eq:delta_0_def}.  This completes the proof of~\eqref{eq:proof_step4_tech2}.

%Again, it suffices to establish the same relation with $f_n(s;\beta)$ instead of $\psi_n(s;\beta)$.
%Let $f(s) = \E \eee^{i s Z_1}$ be the characteristic function of $Z_1$ defined in~\eqref{eq:Z_1_distr}. Then,
%$$
%|f_n(s;\beta )| = \left|f\left(\frac s{\sigma(\beta) \sqrt n}\right)\right|^n.
%$$
%The random variable $Z_1$ has a lattice distribution with width $1$ by Assumption~E (this is the only place in the proof where we use it in an %essential way) and hence $|f(s)|<1-\eps$ for all $s\in [\frac a{\sigma(\beta)}, \pi]$ and some $\eps>0$. Moreover, this holds uniformly in $\beta\in %K$. Hence
%\begin{align*}
%\int_{a\sqrt n \leq  |s|\leq \pi \sigma(\beta)  \sqrt n} |f_n(s;\beta)| \dd s
%&=
%\int_{a\sqrt n \leq  |s|\leq \pi \sigma(\beta)  \sqrt n} \left|f\left(\frac s{\sigma(\beta) \sqrt n}\right)\right|^n \dd s\\
%&\leq 2\pi\sigma(\beta)  \sqrt n (1-\eps)^n.
%\end{align*}
%This implies~\eqref{eq:proof_main_tech1b}.

\vspace*{2mm}
\noindent
\textsc{Step 4.} We show that
\begin{equation}\label{eq:proof_main_tech2}
n^{r/2} \int_{|s|>n^{1/7}} \eee^{-\frac 12 s^2} |V_{r,n}(s;\beta)| \dd s \toas 0.
\end{equation}
By~\eqref{eq:U_tilde_def} we have that for $n\in\N$,
$$
|V_{r,n}(s;\beta )| < M (1+|s|^{4r}),
$$
where $M$ is an a.s.\ finite random variable. This easily implies~\eqref{eq:proof_main_tech2}.

\vspace*{2mm}
Taken together, the results of Steps 1--4 imply~\eqref{eq:tech_main} and thus complete the proof of
Theorem~\ref{theo:asympt_expansion_BRW}.

\subsection{Proof of Proposition~\ref{prop:CLT_global_speed}}
The idea is to take sums in Theorem~\ref{theo:asympt_expansion_BRW} and then to approximate these sums by Riemann integrals.

\vspace*{2mm}
\noindent
\textsc{Step 0.}
By Theorem~\ref{theo:asympt_expansion_BRW} with $r=3$ we have
\begin{equation}\label{eq:proof_berry_esseen_exp}
\frac{L_n(h)}{W_\infty(0) m^n}
=
\frac{\eee^{-\frac 12 x_n^2(h)}}{\sigma(0) \sqrt{2\pi n}}
\left(
1 + \frac{F_1(x_n(h);0)}{W_\infty(0)\sqrt n} + \frac{F_2(x_n(h);0)}{W_\infty(0) n}
%+\frac{F_3(x_n(h);0)}{W_\infty(0) n^{3/2}}
\right)
+ o\left(\frac 1{n^{3/2}}\right) \; \text{a.s.},
\end{equation}
where the $o$-term is uniform in $h\in\Z$ and in the formula for $x_n(h)$ we take $\beta=0$.

\vspace*{2mm}
\noindent
\textsc{Step 1.}
By the standard error term analysis for the trapezoidal rule, for any integers $k_1<k_2$ and every function $f\in C^2[k_1,k_2]$ we have
$$
\sum_{h=k_1}^{k_2} f(h) \,= \, \frac 12 f(k_1) + \frac 12 f(k_2) + \int_{k_1}^{k_2} f(t)\dd t + \frac 1{12} \sum_{h=k_1+1}^{k_2} f''(\xi_h),
$$
where $\xi_h\in [h-1,h]$. Using this formula for the function $f(x_n(h))$ instead of $f(h)$ we obtain
\begin{align*}
\sum_{h=k_1}^{k_2} f(x_n(h))
%& =
%\frac 12 f(x_n(k)) + \int_{-\infty}^k f(x_n(t))\dd t + \frac 1{12 \sigma^2(0) n} \sum_{j=-\infty}^k %f''(x_n(\xi_j))\\
\ &=\
\frac 12 f(x_n(k_1)) + \frac 12 f(x_n(k_2)) + \sigma(0) \sqrt n \int_{x_n(k_1)}^{x_n(k_2)} f(z)\dd z\\
               &\hspace{4.55cm} + \frac {\sum_{h=k_1+1}^{k_2} f''(x_n(\xi_h))}{12 \sigma^2(0) n} .
\end{align*}
In our applications, $f''$ decays exponentially so that we can estimate the error term involving the sum of $f''(\xi_h)$ by $O(1/\sqrt n)$ uniformly in $k_1,k_2$. In particular, uniformly over $|k|<[n^{3/4}]$ we have
\begin{align}
&\sum_{h=-[n^{3/4}]}^k  \eee^{-\frac 12 x_n^2(h)}
=
\frac 12 \eee^{-\frac 12 x_n^2(k)}
+
\sigma(0) \sqrt n \int_{-\infty}^{x_n(k)}  \eee^{-\frac 12 z^2} \dd z + O\left(\frac 1{\sqrt n}\right),\label{eq:berry_esseen_trap1}\\
&\sum_{h=-[n^{3/4}]}^k  \eee^{-\frac 12 x_n^2(h)} F_1(x_n(h);0)
=
\sigma(0) \sqrt n \int_{-\infty}^{x_n(k)} \eee^{-\frac 12 z^2} F_1(z;0)\dd z
+O(1) \;\;\; \text{a.s.},\label{eq:berry_esseen_trap2}\\
&\sum_{h=-[n^{3/4}]}^k \eee^{-\frac 12 x_n^2(h)} F_2(x_n(h);0) = O(\sqrt n) \;\;\; \text{a.s.}, \label{eq:berry_esseen_trap3}
\end{align}
where we dropped the integrals over the range $z<-n^{3/4}$ because of
$$
\int_{|z|\geq [n^{3/4}]}  \eee^{-\frac 12 z^2} \dd z =o(n^{-C}),
\quad
\int_{|z|\geq [n^{3/4}]}  \eee^{-\frac 12 z^2} F_1(z;0) \dd z =o(n^{-C}) \;\;\text{a.s.}
$$
for every $C>0$. Recalling the formula for $F_1(z;0)$, see~\eqref{eq:F_1_BRW}, we obtain
\begin{equation}\label{eq:berry_esseen_int}
\int_{-\infty}^{x_n(k)} \eee^{-\frac 12 z^2} F_1(z;0)\dd z
=
-\frac{\kappa_3(0)W_\infty(0)}{6\sigma^3(0)} \Herm_2(x_n(k)) - \frac{W_\infty'(0)}{\sigma(0)}.
\end{equation}
Let $|k|<[n^{3/4}]$. Taking in~\eqref{eq:proof_berry_esseen_exp} the sum over $h=-[n^{3/4}],\ldots,k$ and using~\eqref{eq:berry_esseen_trap1}, \eqref{eq:berry_esseen_trap2}, \eqref{eq:berry_esseen_trap3}, \eqref{eq:berry_esseen_int}, we obtain
\begin{multline}\label{eq:berry_esseen_cuttof}
\frac 1 {W_\infty(0) m^n} \sum_{h=-[n^{3/4}]}^{k} L_n(h) - \frac {1}{\sqrt{2\pi}} \int_{-\infty}^{x_n(k)} \eee^{-\frac 12 z^2} \dd z =
\\
 \frac {\eee^{-\frac 12 x_n^2(k)}} {\sigma(0) \sqrt {2\pi n} }
\left(\frac 12 - \frac{\kappa_3(0)}{6\sigma^2(0)} (x_n^2(k) -1) - \frac {W_\infty'(0)}{W_\infty(0)}\right) + o\left(\frac 1 {\sqrt n}\right)
\quad\text{a.s.}
\end{multline}
Note that the error term $o(n^{-3/2})$ in~\eqref{eq:proof_berry_esseen_exp} taken $n^{3/4}$ times can be estimated by $o(1/\sqrt n)$.

\vspace*{2mm}
\noindent
\textsc{Step 2.}
It remains to estimate $\sum_{|h|\geq [n^{3/4}]} L_n(h)$.  It is an easy consequence of the Chernoff inequality that for every $C>0$,
$$
\frac 1 {m^n} \sum_{|h| \geq  [n^{3/4}]} \E L_n(h) = o(n^{-(C+2)}).
$$
By the Borel--Cantelli lemma and the Markov inequality,
\begin{equation}\label{eq:berry_esseen1}
\frac 1 {m^n} \sum_{|h| \geq  [n^{3/4}]} L_n(h) = o(n^{-C}) \quad \text{a.s.}
\end{equation}
This completes the proof of Proposition~\ref{prop:CLT_global_speed}.
%A similar bound holds for the term on the right-hand side of~\eqref{eq:proof_berry_esseen_exp}:
%\begin{equation}\label{eq:berry_esseen2}
%\sum_{|h| > n^{3/5}} \eee^{-\frac 12 x_n^2(h)} F_j(x_n(h);0)  = o(n^{-C}), \quad j=0,1,2.
%\end{equation}

\subsection{Proof of Theorem~\ref{theo:L_n_lim_distr_2_new}}
Recall that, in the first part of the theorem,
$k_n$ is an integer sequence such that for some $\alpha\in\R$,
$$
k_n = \varphi'(\beta) n + \alpha \sigma(\beta) \sqrt {n} + o(\sqrt n), \quad n\to\infty.
$$
Note that $x_n := x_n(k_n) = \alpha + o(1)$; see~\eqref{eq:x_n_k_def}.
By Corollary~\ref{cor:asympt_expansion_BRW_without_expectation} with $r=1$, we have
\begin{equation} \label{eq:tech_expansion_r1_without_E}
\eee^{\beta k_n -\varphi(\beta)n} \sqrt{2\pi n} \,\sigma(\beta) L_n^{\circ}(k_n)
=
\eee^{-\frac12 x_n^2}\cdot \frac{1}{\sqrt n} F_1^\circ(x_n;\beta) + o\left(\frac 1{\sqrt n}\right) \quad \text{a.s.},
\end{equation}
Here, $F_1^\circ$ is given by~\eqref{eq:F_1_circ_BRW}, hence
\begin{equation}\label{eq:tech_F1_circ_first}
F_1^\circ(x_n;\beta) = \frac{x_n}{\sigma(\beta)} W_\infty'(\beta) = \frac{\alpha}{\sigma(\beta)} W_\infty'(\beta) + o(1) \quad\text{a.s.}
\end{equation}
Inserting~\eqref{eq:tech_F1_circ_first} into~\eqref{eq:tech_expansion_r1_without_E}, we obtain
$$
\eee^{\beta k_n - \varphi(\beta)n} \sqrt{2\pi n}\, \sigma(\beta) L_n^{\circ}(k_n)
=
\frac{\alpha}{\sigma(\beta)\sqrt n} \eee^{-\frac12 \alpha^2} W_\infty'(\beta) + o\left(\frac 1{\sqrt n}\right) \quad \text{a.s.}.
$$
This proves the first part of Theorem~\ref{theo:L_n_lim_distr_2_new}.

Now let $k_n$ be an integer sequence such that
$$
k_n = \varphi'(\beta) n + c_n, \text{ where } c_n=o(\sqrt n).
$$
Note that
\begin{equation}\label{eq:tech_x_n_k}
x_n := x_n(k_n) = \frac{c_n}{\sigma(\beta) \sqrt n} = o(1).
\end{equation}
%Recall that in the setting of both theorems  we have $W_\infty(0)=1$.
By Corollary~\ref{cor:asympt_expansion_BRW_without_expectation} with $r=2$, we have
\begin{align}
\lefteqn{\eee^{\beta k_n - \varphi(\beta) n}\sqrt{2\pi n}\, \sigma(\beta)  L_n^{\circ}(k_n)}
\label{eq:tech_expansion_r2_without_E}\\
&=
\eee^{-\frac12 x_n^2} \left[\frac{1}{\sqrt n} F_1^\circ(x_n;\beta) + \frac 1n F_2^\circ(x_n;\beta)\right] + o\left(\frac 1n\right) \quad \text{a.s.},\notag
\end{align}
where $F_1^\circ$ and $F_2^\circ$ are given by equations~\eqref{eq:F_1_circ_BRW} and~\eqref{eq:F_2_circ_BRW}.
%The asymptotics of $F_1^\circ$ and $F_2^\circ$ can be computed by dropping all terms involving $W_\infty(0)$ in~\eqref{eq:tech_F1}, \eqref{eq:tech_F2}, \eqref{eq:tech_C}.
Using these equations together with~\eqref{eq:tech_x_n_k}, we see that the asymptotic expressions for $F_1^\circ$ and $F_2^\circ$ look as follows:
\begin{align}
\frac{1}{\sqrt n} F_1^\circ(x_n;\beta)
&=
\frac {c_n} {\sigma^2(\beta) n} W_\infty'(\beta),\label{eq:tech_F1_circ}\\
\frac 1n F_2^\circ(x_n; \beta)
&=
\frac 1n \left(\frac{\kappa_3(\beta)}{2\sigma^4(\beta)}W_\infty'(\beta) - \frac{1}{2\sigma^2(\beta)}W_\infty''(\beta)\right) +  o\left(\frac 1n\right)\quad \text{a.s.} \label{eq:tech_F2_circ}
\end{align}
Here we used that $\Herm_4(0)=3$ and $\Herm_2(0)=-1$; see~\eqref{eq:Herm1}, \eqref{eq:Herm2}. Also, note that by~\eqref{eq:tech_x_n_k},
\begin{equation}\label{eq:tech_Taylor_circ}
\eee^{-\frac 12 x_n^2}
=
1+o(1).
%1 - \frac {c_n^2}{2\sigma^2 n} + O\left(\frac {c_n^4}{n^2}\right).
\end{equation}
Now we consider two cases.

\vspace*{2mm}
\noindent
\textsc{Case 1.}
Let $c_n\to +\infty$ but $c_n=o(\sqrt n)$.
Inserting~\eqref{eq:tech_F1_circ}, \eqref{eq:tech_F2_circ}, \eqref{eq:tech_Taylor_circ} into~\eqref{eq:tech_expansion_r2_without_E}, we obtain
$$
\eee^{\beta k_n - \varphi(\beta) n}\sqrt{2\pi n}\, \sigma(\beta) L_n^{\circ}(k_n) =
\frac{c_n}{\sigma^2(\beta) n}W_\infty'(\beta) +  o\left(\frac{c_n}{n}\right) \quad \text{a.s.}
$$
This proves the second part Theorem~\ref{theo:L_n_lim_distr_2_new}.

\vspace*{2mm}
\noindent
\textsc{Case 2.}
Let $c_n=O(1)$.
Inserting~\eqref{eq:tech_F1_circ}, \eqref{eq:tech_F2_circ}, \eqref{eq:tech_Taylor_circ} into~\eqref{eq:tech_expansion_r2_without_E}, we obtain
\begin{multline}
\eee^{\beta k_n - \varphi(\beta)n} \sqrt{2\pi n}\, \sigma(\beta) L_n^{\circ}(k_n) \\
=
\frac 1n\left( \left(\frac {c_n} {\sigma^2(\beta)} + \frac{\kappa_3(\beta)}{2\sigma^4(\beta)}\right) W_\infty'(\beta) - \frac{1}{2\sigma^2(\beta)}W_\infty''(\beta)\right) + o\left(\frac 1n\right) \quad\text{a.s.}
\end{multline}
This proves the third part of Theorem~\ref{theo:L_n_lim_distr_2_new}.

\subsection{Proof of Theorem~\ref{theo:mode_height}}\label{subsec:proof_mode_width}
We start by proving the formula for the mode $u_n$ stated in the first part of the theorem.
We search for $k\in\Z$ maximizing the occupation number $L_n(k)$.

\vspace*{2mm}
\noindent
\textsc{Step 0.}
We write $k=k_n(a)=\varphi'(0)n + a$, where $a$ is subject to the restriction $a\in \Z - \varphi'(0) n$.  Let $\beta=0$ and  $\sigma^2 = \varphi''(0)$.  With this notation, we have
$$
x_n := x_n(k)= \frac{a}{\sigma\sqrt {n}}.
$$
By Theorem~\ref{theo:asympt_expansion_BRW} with $r=2$, we have
\begin{equation}\label{eq:tech_expansion_r2}
m^{-n}\sqrt{2\pi n}\, \sigma L_n(k)
=
\eee^{-\frac12 x_n^2} \left(W_\infty(0) + \frac{1}{\sqrt n} F_1(x_n;0) + \frac 1n F_2(x_n;0)\right) + o\left(\frac 1n\right) \quad \text{a.s.},
\end{equation}
where $F_1$ and $F_2$ are given by equations~\eqref{eq:F_1_BRW} and~\eqref{eq:F_2_BRW}.
The $o$-term is uniform over $a\in \Z - \varphi'(0) n$.

\vspace*{2mm}
\noindent
\textsc{Step 1.} Let us first assume that $|a| < n^{1/4-\eps}$ for some small $\eps>0$. Later we will show that, asymptotically, the values of $a$ outside this range have no chance to be the mode. Let us agree that all estimates will be uniform over $|a| < n^{1/4-\eps}$.
Recalling~\eqref{eq:F_1_BRW} and~\eqref{eq:F_2_BRW} and noting that $x_n^3 = o(n^{-1/2})$ and $x_n=o(1)$ we obtain
\begin{align}
\frac{1}{\sqrt n} F_1(x_n;0)
&=
\frac a n \left( \frac 1 {\sigma^2} W_\infty'(0) - \frac{\kappa_3(0)}{2\sigma^4} W_\infty(0)\right) + o\left(\frac 1n\right)\quad\text{a.s.},\label{eq:tech_F1}\\
\frac 1n F_2(x_n; 0)
&=
\frac 1n C +  o\left(\frac 1n\right)\quad \text{a.s.} \label{eq:tech_F2},
\end{align}
where
\begin{equation}\label{eq:tech_C}
C= \left(\frac{\kappa_4(0)}{8\sigma^4} - \frac{5\kappa_3^2(0)}{24\sigma^6}\right)W_\infty(0) + \frac{\kappa_3(0)}{2\sigma^4}W_\infty'(0) - \frac{1}{2\sigma^2}W_\infty''(0)
\end{equation}
and we used that $\Herm_4(0)=3$, $\Herm_6(0)=-15$, $\Herm_2(0)=-1$; see~\eqref{eq:Herm1}, \eqref{eq:Herm2}.  Noting that $x_n^4=o(1/n)$, we obtain by Taylor's formula that
\begin{equation}\label{eq:tech_Taylor}
\eee^{-\frac 12 x_n^2}
=
1 - \frac {a^2}{2\sigma^2 n} + o\left(\frac 1n\right).
\end{equation}
Inserting~\eqref{eq:tech_F1}, \eqref{eq:tech_F2}, \eqref{eq:tech_Taylor} into~\eqref{eq:tech_expansion_r2}, we obtain
\begin{align}
\lefteqn{m^{-n}\sqrt{2\pi n}\, \sigma L_n(k)}\label{eq:tech_expansion1}\\
&=
W_\infty(0) + \frac 1n \left(a\left(\frac 1 {\sigma^2} W_\infty'(0) - \frac{\kappa_3(0)}{2\sigma^4} W_\infty(0)\right) - \frac {a^2}{2\sigma^2} W_\infty(0) + C\right) +  o\left(\frac 1n\right) \quad \text{a.s.} \notag
\end{align}
Differentiation shows that the $\argmax$ of the quadratic function in the brackets is given by
\begin{equation}\label{eq:a_*}
a_* = \frac{W_\infty'(0)}{W_\infty(0)} - \frac{\kappa_3(0)}{2\sigma^2}.
\end{equation}
However, we have the restriction $a\in \Z - \varphi'(0) n$. Using~\eqref{eq:tech_expansion1} we easily obtain that
\begin{equation}\label{eq:L_n_k_plus_1_L_n_k}
m^{-n}\sqrt{2\pi n}\, \sigma (L_n(k+1) - L_n(k)) = \frac {W_\infty(0)} {\sigma^2 n} \left(a_* - \frac 12 -a \right) + o\left(\frac 1n\right).
\end{equation}
It follows that there is an a.s.\ finite random variable $N_1$ such that for $n>N_1$, the expression in~\eqref{eq:L_n_k_plus_1_L_n_k} is negative for $a > a_*-\frac 12$ and positive for $a<a_*-\frac 12$.
Hence, for $n>N_1$, the mode (computed over the range $|a|<n^{1/4-\eps}$) is equal to one of the numbers
$\lfloor u_n^*\rfloor$ or $\lceil u_n^*\rceil$, where $u_n^* = \varphi'(0)n + a_*$.

However, note that we do \textit{not} claim that the mode is equal to $u_n^*$ rounded to the nearest integer! Indeed, if $u_n^*$ is so close to a half-integer that the first term on the right-hand side of~\eqref{eq:L_n_k_plus_1_L_n_k} is smaller than the error term $o(1/n)$, then we cannot tell whether $\lfloor u_n^*\rfloor$ or $\lceil u_n^* \rceil$ is the mode without considering further terms in the asymptotic expansion.

\vspace*{2mm}
\noindent
\textsc{Step 2.} Inserting into~\eqref{eq:tech_expansion1} any $a=O(1)$ we obtain that
\begin{equation}\label{eq:tech_estimate1}
m^{-n}\sqrt{2\pi n}\, \sigma L_n(k)  = W_\infty(0) + O\left(\frac 1n\right) \quad \text{a.s.}
\end{equation}
Take any $B>0$. In order to complete the proof of Theorem~\ref{theo:mode_height} (a) it suffices to show that %~\eqref{eq:tech_estimate1} cannot hold for $|a|>n^{1/4}$.
\begin{equation}\label{eq:mode_not_large}
W_\infty(0) - m^{-n}\sqrt{2\pi n}\, \sigma  L_n(k) > \frac Bn
\end{equation}
for all $n>N_2$ and $|a|\geq n^{1/4-\eps}$.  Here, we denote by $N_2,N_3,\ldots$ a.s.\ finite random variables. The proof of~\eqref{eq:mode_not_large} will be given in Steps 3--5 below.

\vspace*{2mm}
\noindent
\textsc{Step 3.} Let $|a|>\sigma \sqrt n$. Then $|x_n|>1$. By Theorem~\ref{theo:asympt_expansion_BRW} with $r=0$ we have
\begin{equation}\label{eq:tech_expansion_big_a1}
m^{-n}\sqrt{2\pi n}\, \sigma L_n(k)
=
\eee^{-\frac12 x_n^2} W_\infty(0) + o(1) \quad \text{a.s.}
\end{equation}
The $o$-term is uniform, hence
$$
W_\infty(0)-m^{-n}\sqrt{2\pi n}\, \sigma L_n(k) =
W_\infty(0) (1-\eee^{-\frac12 x_n^2})  - o(1)
>
\frac 13 W_\infty(0) > \frac{B}{n}
$$
for $n>N_3$, where we used the inequality $1-\eee^{-\frac 12} > \frac 13$  and the fact that $W_\infty(0)>0$ a.s.

\vspace*{2mm}
\noindent
\textsc{Step 4.} Let $|a|\leq \sigma \sqrt n$ but $|a|>\sigma n^{3/8 + \eps}$ for some $\eps>0$. This means that $|x_n|\leq 1$ but $|x_n| > n^{-1/8 + \eps}$. By Theorem~\ref{theo:asympt_expansion_BRW} with $r=1$ we have
\begin{equation}\label{eq:tech_expansion_big_a2}
m^{-n}\sqrt{2\pi n}\, \sigma L_n(k)
=
\eee^{-\frac12 x_n^2} \left(W_\infty(0) + \frac 1 {\sqrt n} F_1(x_n;0)\right) + o\left(\frac 1{\sqrt n}\right) \quad \text{a.s.}
\end{equation}
Recalling~\eqref{eq:F_1_BRW} and using the condition $|x_n|\leq 1$, we obtain that $F_1(x_n; 0) = O(1)$. Hence
$$
W_\infty(0)-m^{-n}\sqrt{2\pi n}\, \sigma L_n(k)
=
W_\infty(0) (1-\eee^{-\frac12 x_n^2}) -  O\left(\frac 1{\sqrt n}\right)
> \frac {x_n^2}3  W_\infty(0) - O\left(\frac 1{\sqrt n}\right),
$$
where we used the inequality $1- \eee^{-y/2} > y/3$ valid for $y\in [0,1]$.  By the inequality $x_n^2 > n^{-\frac 14 + 2\eps}$, the right-hand side is bigger than $B/n$ for $n>N_4$.

\vspace*{2mm}
\noindent
\textsc{Step 5.} Finally, let $|a| \leq \sigma n^{3/8 + \eps}$ but $|a| > \sigma n^{1/4-2\eps}$. Equivalently, $|x_n|\leq n^{-1/8 + \eps}$ but $|x_n| > n^{-1/4-2\eps}$.
By Theorem~\ref{theo:asympt_expansion_BRW} with $r=2$ we have
\begin{equation}\label{eq:tech_expansion_big_a3}
m^{-n}\sqrt{2\pi n}\, \sigma L_n(k)
=
\eee^{-\frac12 x_n^2} \left(W_\infty(0) + \frac 1 {\sqrt n} F_1(x_n;0) +\frac 1n F_2(x_n;0)\right) + o\left(\frac 1n\right) \quad \text{a.s.}
\end{equation}
The above assumptions on $x_n$ together with~\eqref{eq:F_1_BRW} and~\eqref{eq:F_2_BRW} imply that
$$
\frac{1}{\sqrt n} F_1(x_n;0) =O(n^{-\frac 58 + \eps}),
\quad
\frac{1}{n} F_2(x_n;0) = O\left(\frac 1n\right).
$$
It follows that for sufficiently small $\eps>0$, say $\eps<1/20$,
$$
W_\infty(0)-m^{-n}\sqrt{2\pi n} \sigma L_n(k)
=
W_\infty(0) (1-\eee^{-\frac12 x_n^2}) -  O\left(n^{-\frac 58 + \eps}\right)
> \frac {x_n^2}3  W_\infty(0) - o\left(x_n^2\right),
$$
where we again used the inequality $1- \eee^{-y/2} > y/3$ valid for $y\in [0,1]$. By the inequality $x_n^2 > n^{-1/2 - 4\eps}$, the right-hand side is bigger than $B/n$ for $n>N_5$.
This completes the proof of~\eqref{eq:mode_not_large}.

\vspace*{2mm}
\noindent
\textsc{Step 6.} Finally, we prove Theorem~\ref{theo:mode_height}\,(b). We are interested in the height $M_n=L_n(u_n)$, where $u_n$ is the mode equal either to $\lfloor u_n^*\rfloor$ or to $\lceil u_n^*\rceil$. The idea is to insert both numbers into expansion~\eqref{eq:tech_expansion1}. Using~\eqref{eq:tech_expansion1} with $a=a_*+\theta$,  $|\theta|\leq 1$,   and recalling that $a_*$ is given by~\eqref{eq:a_*}, we obtain
\begin{align*}
\frac{\sqrt{2\pi n}\, \sigma\, L_n(u_n^* + \theta)}{m^{n}}
&=
W_\infty(0) + \frac {W_\infty(0)}{2\sigma^2 n} \left(2(a_*+\theta)a_* - (a_*+\theta)^2\right) + \frac Cn +o\left(\frac 1n\right)\\
&=
W_\infty(0) + \frac {W_\infty(0)}{2\sigma^2 n} \left(a_*^2 -\theta^2\right) + \frac Cn  + o\left(\frac 1n\right) \text{ a.s.}
\end{align*}
Recalling the formula for $C$, see~\eqref{eq:tech_C}, and after some straightforward transformations, we obtain
$$
\frac{\sqrt{2\pi n}\, \sigma\, L_n(u_n^* + \theta)}{W_\infty(0) \, m^{n}}
=
1 - \frac{1}{2\sigma^2 n}
\left(\frac{\kappa_3^2(0)}{6\sigma^4} - \frac{\kappa_4(0)}{4\sigma^2} + \theta^2
%+ \frac{W_\infty'^2(0) - W_\infty''(0)W_\infty(0)}{W_\infty^2(0)}
+(\log W_\infty)''(0)
\right) + o\left(\frac 1n \right)
\;\;\; \text{a.s.}
$$
Now, the mode $u_n$ is equal to $u_n^*+\theta$ with $\theta$ being either $\lfloor u_n^* \rfloor - u_n^*$ or $\lceil u_n^* \rceil - u_n^*$. It follows that
$$
\frac{\sqrt{2\pi n}\, \sigma\, L_n(u_n)}{W_\infty(0) \, m^{n}}
=
1 - \frac{1}{2\sigma^2 n}
\left(\frac{\kappa_3^2(0)}{6\sigma^4} - \frac{\kappa_4(0)}{4\sigma^2} + \theta_n^2
%+ \frac{W_\infty'^2(0) - W_\infty''(0)W_\infty(0)}{W_\infty^2(0)}
+(\log W_\infty)''(0)
\right) + o\left(\frac 1n \right)
\;\;\; \text{a.s.},
$$
where $\theta_n$ is the smaller one of the numbers $u_n^* - \lfloor u_n^* \rfloor$ or $\lceil u_n^* \rceil - u_n^*$.

\subsection*{Acknowledgement} The authors are grateful to Henning Sulzbach and Alexander Marynych for useful discussions and to the unknown referee for comments which improved the exposition.

\bibliographystyle{plainnat}
\bibliography{edgeworth_brw_rev}

\end{document}